\documentclass{article}

\usepackage{amsmath}
\usepackage{amssymb}

\usepackage[amsmath,thmmarks]{ntheorem}

\usepackage[2cell,v2]{xy}
\UseTwocells
\UseHalfTwocells

\newcommand{\tmem}[1]{{\em #1\/}}
\newcommand{\tmmathbf}[1]{\mathbf{#1}}
\newcommand{\tmop}[1]{{#1}}

\newcommand{\mathbbm}[1]{\mathbb{#1}}

\newcommand{\dueto}[1]{\textup{{(#1) }}}

\newcommand{\leirom}{\renewcommand{\labelenumi}{\textit{(\roman{enumi})}}}
\newcommand{\leialph}{\renewcommand{\labelenumi}{(\alph{enumi})}}
\newcommand{\leiarab}{\renewcommand{\labelenumi}{(\arabic{enumi})}}

\newcommand{\leiialph}{\renewcommand{\labelenumii}{(\alph{enumii})}}
\newcommand{\leiiarab}{\renewcommand{\labelenumii}{(\arabic{enumii})}}

\newenvironment{enumerateroman}
{\leirom \begin{enumerate}}
{\end{enumerate} \leiarab}

\newenvironment{itemizeminus}
  {\begin{itemize}}{\end{itemize}}

\newenvironment{enumeratenumeric}{\begin{enumerate}}{\end{enumerate}}

\newenvironment{enumeratealpha}
{\leialph \leiialph \begin{enumerate}}
{\end{enumerate} \leiarab \leiiarab}

\newcommand{\tmstrong}[1]{\textbf{#1}}

\usepackage[amsmath,thmmarks]{ntheorem}

\theoremstyle{plain} 
\theoremheaderfont{\normalfont\bfseries}
\theorembodyfont{\slshape} 
\theoremseparator{.}
\theoremsymbol{} 
\newtheorem{theorem}{Theorem} 
\newtheorem{proposition}{Proposition} 
\newtheorem{lemma}{Lemma}
\newtheorem{corollary}{Corollary}

\newtheorem*{itheorem}{Theorem}

\theoremstyle{plain} 
\theoremheaderfont{\normalfont\bfseries}
\theorembodyfont{\slshape} 
\theoremseparator{.}
\theoremsymbol{}
\newtheorem{definition}{Definition}

\theoremstyle{plain} 
\theoremheaderfont{\slshape}
\theorembodyfont{\upshape} 
\theoremseparator{.}
\theoremsymbol{\ensuremath{\lozenge}}

\newtheorem{remark}{Remark}
\newtheorem{example}{Example}
\newtheorem{notation}{Notation}

\theoremheaderfont{\slshape}
\theorembodyfont{\upshape} 
\theoremstyle{nonumberplain} 
\theoremseparator{.} 
\theoremsymbol{\ensuremath{\triangleleft}} 
\newtheorem{proof}{Proof} 

\theoremheaderfont{\slshape}
\theorembodyfont{\upshape} 
\theoremstyle{nonumberplain} 
\theoremseparator{\hspace{-.5ex}} 
\theoremsymbol{\ensuremath{\triangleleft}}

\theoremstyle{empty} 

\theoremsymbol{} 
\theoremseparator{} 
\theorembodyfont{\slshape}
\theoremprework{\begin{quote}} 
\theorempostwork{\end{quote}\medskip}

\newcommand{\longrightarrowlim}{\mathop{\longrightarrow}\limits}

\newdir{ >}{{}*!/-10pt/\dir{>}}

\begin{document}

\title{Sheaves of ordered spaces and interval theories}

\author{Krzysztof Worytkiewicz\\
AGH\\
Al. Mickiewicza 30\\
30-059 Krak\'ow, Poland }

\maketitle

\begin{abstract}
  We study the homotopy theory of locally ordered spaces, that is manifolds
  with boundary whose charts are partially ordered in a compatible way. Their
  category is not particularly well-behaved with respect to colimits. However,
  this category turns out to be a certain full subcategory of a topos of
  sheaves over a simpler site. The ambient topos makes available some general
  homotopical machinery. 
\end{abstract}

\section{Introduction}

{\noindent}It has been of interest for some time to
consider topological spaces where paths are made irreversible, globally or
locally. Such artifacts are well suited to model the behavior of interacting
computational processes, in a way which captures the flow of time. A typical
setup involves topological spaces interacting with order structures.
Computational paths are modeled by continuous locally non-decreasing maps.
Meaningful homotopies among such paths are the non-decreasing ones, that is
those which respect the flow of time. Such homotopies are called
{\tmem{directed}} in the litterature {\cite{FRG}}.

{\noindent}There are many variants of the notion of directed homotopy. In this paper we
study two important variations, namely ``Di-homotopy'' {\cite{FRG}} on one
hand and ``D-homotopy'' {\cite{grandis2}} on the other. Di-homotopy is much
like the usual homotopy in the category of topological spaces, in the sense
that the standard topological interval is used. However, it takes place in the
category of locally ordered spaces and so is equipped with the discrete order
while all the maps involved, including the homotopies themselves, are locally
non-decreasing. This is to be contrasted with D-homotopy where the standard
topological interval is equipped with the natural order from the start. It is
to be said that D-homotopy as studied in the literature occurs in settings
distinct to the present one, namely in so-called D-spaces with better
categorical properties. However, this is achieved at a price: it has to be
distinguished between directed and undirected paths in a way which may seem
arbitrary. D-homotopy makes nonetheless good sense also when living in the
category of locally ordered spaces. The relationship between these two notions
of directed homotopy has been a recurrent question for some time. The present
work can be seen as an effort to give a homotopy theoretic answer to this
question.

{\noindent}The paper is organized as follows. Section \ref{sec:locored} contains
definitions and some facts about our notion of elementary partially ordered
spaces viz. {\tmem{epo-spaces}} and our notion of locally partially ordered
spaces viz. {\tmem{local epo-spaces}}.

{\noindent}In section \ref{sec:losheaves} we introduce the site $( \mathbb{P}, \tau)$ of
epo-spaces and exhibit the category $\mathbb{L}$ of local epo-spaces as a full
subcategory of the topos of sheaves $\tmop{Sh} ( \mathbb{P}, \tau)$.

% {\noindent}\tmtextbf{Theorem. }\tmtextit{Let $\mathbb{L}$ be the category of
% local epo-spaces. The embedding $h_{\mathbb{P} } : \mathbb{L} \longrightarrow
% \tmop{Sh} \left( \mathbb{P}, \tau \right)$ given by restriction of the Yoneda
% functor
% \[ X \longmapsto \mathbb{L} (-, X)_{| \mathbb{P}} \]
% is full and faithful.}{\hspace*{\fill}}{\medskip}

\begin{itheorem}
Let $\mathbb{L}$ be the category of
local epo-spaces. The embedding $h_{\mathbb{P} } : \mathbb{L} \longrightarrow
\tmop{Sh} \left( \mathbb{P}, \tau \right)$ given by restriction of the Yoneda
functor
\[ X \longmapsto \mathbb{L} (-, X)_{| \mathbb{P}} \]
is full and faithful.
\end{itheorem}

{\noindent}We further characterize those sheaves which are local epo-spaces,
ultimately in terms of {\tmem{\'etale}} dimaps. As might be expected, we call
a dimap \'etale if the underlying continuous map is a local homeomorphism.
Such dimaps are obviously stable under pullbacks. \'Etale dimaps lead to the
notion of $\mathbb{P}$-locality. Namely, a morphism of sheaves $\alpha : F
\longrightarrowlim G$ is {\tmem{$\mathbb{P}$-local}} if pulling back any
morphism
\[ h_{\mathbb{P}} (X) \longrightarrow G \]
from a local epo-space $h_{\mathbb{P} } (X)$ to $G$ along $\alpha$ yields
another local epo-space $h_{\mathbb{P} } (Y)$

% \begin{center}
%   \epsfig{file=rev2-1.eps}
% \end{center}

\begin{center}
$
\xy
{\ar@{->}^{\pi_1} (0,0)*+{h_\mathbb{P}(Y)} ; (25,0)*+{h_\mathbb{P}(X)}};
{\ar@{->}_{\pi_2} (0,0)*+{h_\mathbb{P}(Y)} ; (0,-20)*+{F}};
{\ar@{->}_\alpha (0,-20)*+{F} ; (25,-20)*+{G}};
{\ar@{->} (25,0)*+{h_\mathbb{P}(X)} ; (25,-20)*+{G}};
{\ar@{-} (5,-5)*{} ; (5,-2.5)};
{\ar@{-} (5,-5)*{} ; (2.5,-5)};
\endxy
$
\end{center}

{\noindent}and, moreover, if the canonical morphism $\pi_1 : h_{\mathbb{P} }
(Y) \longrightarrow h_{\mathbb{P} } (X)$ is induced by an \'etale dimap.

%\tmtextbf{Theorem. }\tmtextit{

\begin{itheorem}
The following are equivalent:
\begin{enumerateroman}
  \item a sheaf $L \in \tmop{Sh} ( \mathbb{P}, \tau)$ is a local epo-space;
  
  \item there is a family
  \[ \left( \kappa_i : h_{\mathbb{P} } (U_i) \longrightarrow L \right)_{i \in
     I} \]
  of $\mathbb{P}$-local monos such that the canonical morphism
  \[ [\kappa_i]_{i \in I} : \coprod_{i \in I} h_{\mathbb{P} } (U_i)
     \longrightarrow L \]
  is an epi.
\end{enumerateroman}
\end{itheorem}

%}{\hspace*{\fill}}{\medskip}

{\noindent}In section \ref{sec:intervals} we briefly review the material of
{\cite{cisinski-topos}} about interval-based model structures in Grothendieck
topoi. The weak equivalences of such model structures are given by
contravariant action on quotients of certain homsets while cofibrations are
always monos. It is in fact a (very) far-reaching generalization of the
classical work of Gabriel and Zisman {\cite{gz}}. We then build on this
material by introducing a natural notion of morphism of intervals. Given such
a morphism, there are in particular two model structures on the same topos,
induced by the source respectively the target interval. We investigate the
relationship between these model structures under additional hypotheses. Our
main observation can be summarized as follows.

%\tmtextbf{Theorem. }\tmtextit{

\begin{itheorem}
Let $\mathcal{I}$ and $\mathcal{I}'$
be intervals in a topos and $\mathcal{W}_{\mathcal{I}}$ respectively
$\mathcal{W}_{\mathcal{I}'}$ be the classes of weak equivalences in the
induced model structures. Suppose $\iota : \mathcal{I} \longrightarrow
\mathcal{I}'$ is a morphism of intervals. Then
\[ \mathcal{W}_{\mathcal{I}} \subseteq \mathcal{W}_{\mathcal{I}'} \]
if $\iota$ is a section-wise $\mathcal{I}$-weak equivalence and
\[ \mathcal{W}_{\mathcal{I}'} \subseteq \mathcal{W}_{\mathcal{I}} \]
if $\iota$ is a section-wise $\mathcal{I}'$-weak
equivalence.
\end{itheorem}

%}{\hspace*{\fill}}{\medskip}

In section \ref{sec:diho}, we apply this machinery to compare the homotopy
theories given by the two mentioned notions of directed homotopy.

%\tmtextbf{Theorem. }\tmtextit{

\begin{itheorem}
Let $\mathcal{I}_d$ be the interval
in $\tmop{Sh} ( \mathbb{P}, \tau)$ given by the discrete order on $[0, 1]$ and
$\mathcal{I}_D$ be the interval in $\tmop{Sh} ( \mathbb{P}, \tau)$ given by
the natural order on $[0, 1]$. Then $\mathcal{W}_{\mathcal{I}_D} \subseteq
\mathcal{W}_{\mathcal{I}_d}$ and
\[ \tmop{id} : \left( \tmop{Sh} ( \mathbb{P}, \tau), \mathcal{I}_D \right)
   \longrightarrow \left( \tmop{Sh} ( \mathbb{P}, \tau), \mathcal{I}_d \right)
   \]
is a left Quillen functor.
\end{itheorem}

%\section*{Acknowledgments}

{\noindent}The author acknowledges Peter Bubenik, Eric Goubault, Andr\'e Joyal, Sanjeevi Krishnan 
and Krzysztof Ziemia\'nski for inspiring discussions. He thanks the anonymous referee for pointing
out an error and suggesting a fix, as well as for the general helpful comments and suggestions. 
The author's debt to Rick Jardine
deserves a sentence of its own. The author wishes to thank the MSRI and the
Fields Institute for their hospitality. This research was performed by far and
large at these two institutions where the author was detached from the
Department of Mathematics at the University of Western Ontario. Finally, the author wishes
to thank AGH's Department of Computer Science for support with respect to the Polonium grant.

{\tableofcontents}

\setlength \parindent {0cm}

\section{Locally ordered spaces \label{sec:locored}}

\subsection{\label{no:order-atlases}Atlases}

\begin{definition}
  \label{def:po-space}An {\tmem{epo-space}} is a pair $\left( U, \preccurlyeq
  \right)$ where
  \begin{itemizeminus}
    \item  $U$ is a topological space homeomorphic to an open set of the upper
    half-space
    \[ \mathbbm{H}^n  \overset{\tmop{def} .}{=} \{(x_1, \ldots, x_n) \in
       \mathbb{R}^n | x_n \geqslant 0\} \]
    for some $n \in \mathbb{N}$;
    
    \item $\preccurlyeq \subseteq U \times U$ is a partial order on $U$ which
    is closed in the product topology.
  \end{itemizeminus}
\end{definition}

``Epo-space'' stands for ``elementary partially-ordered space''. The notion of
epo-space is a particular case of the notion of {\tmem{po-space}}
{\cite{FRG}}.

%\tmtextbf{Notation. }\tmtextit{

\begin{notation}
  \label{def:epo-space}An epo-space $\left( U, \preccurlyeq \right)$
  shall be denoted $U$ if the order is understood from context.
\end{notation}

%}{\hspace*{\fill}}{\medskip}

\begin{definition}
  Let $X$ be a topological space.
  \begin{enumeratenumeric}
    \item A {\tmem{chart}} on $X$ is an open subset $U \subseteq X$ which is
    an \ epo-space.
    
    \item Charts $U$ and $U'$ on $X$ are {\tmem{compatible}} if
    \[ \preccurlyeq^U_{|_{U \cap U'}} = \preccurlyeq^{U'}_{|_{U \cap U'}} \]
    \item An atlas on $X$ is an open covering $\left( U_i \right)_{i \in I}$
    of $X$ such that
    \begin{enumerateroman}
      \item $U_i$ is a chart for each $i \in I$;
      
      \item $U_i$ and $U_j$ are compatible for each $(i, j) \in I \times I$.
    \end{enumerateroman}
  \end{enumeratenumeric}
\end{definition}

%\tmtextbf{Notation. }\tmtextit{

\begin{notation}
We write $U \asymp V$ when $U$ and
$V$ are compatible charts. $\tmop{At} (X)$ stands for the collection of
atlases on $X$.
\end{notation}

%}{\hspace*{\fill}}{\medskip}

\subsection{Local epo-spaces}

\begin{definition}
  A {\tmem{local epo-space}} $\left( X, \left( U_i \right) \right)$ consists
  of a topological space $X$ and an atlas $\left( U_i \right)$. 
\end{definition}

\begin{definition}
  A continuous map $f : X \longrightarrow Y$ is locally non-decreasing with
  respect to atlases $\left( U_i \right) \in \tmop{At} (X)$ and to $\left( V_j
  \right) \in \tmop{At} (Y)$ if
  \[ f_{|_{f^{- 1} (V_j) \cap U_i}} : f^{- 1} (V_j) \cap U_i \longrightarrow
     V_j \]
  is non-decreasing for all $(i, j) \in I \times J$.
\end{definition}

\begin{definition}
  Let $\left( X, \left( U_i \right) \right)$ and $\left( Y, \left( V_j \right)
  \right)$ be local epo-spaces. A {\tmem{dimap}} $f : \left( X, \left( U_i
  \right) \right) \rightarrow \text{$\left( Y, \left( V_j \right) \right)$}$
  is a locally non-decreasing continous map.
\end{definition}

\begin{remark}
  An epo-space is a local epo-space equipped with a one-chart atlas. A dimap
  among epo-spaces is just a continuous non-decreasing map. Notice also, that
  the underlying topological space of a local epo-space is a topological
  manifold with boundary.
\end{remark}

\begin{example} \label{exa:1}
  The unit interval $ \left[ 0, 1 \right] \subset \mathbb{R}$ is a manifold
  with non-empty boundary. The discrete order on $[0, 1]$ gives rise the to
  epo-space $\Delta_d$ while the natural order produces the epo-space
  $\Delta_D$.
\end{example}

\begin{example}
  Consider the unit circle $S^1$ and let
  \[ C_{\varepsilon, \varphi}  \overset{\tmop{def} .}{=} \{e^{i \theta} |
     \varphi - \varepsilon \leqslant \theta \leqslant \varphi + \varepsilon\}
  \]
  Then $\left( C_{\pi / 2, \pi / 2}, C_{3 \pi / 2, \pi / 2}  \right)$ is an
  atlas on $S^1$, with the order on the charts being (say) counterclockwise.
  The corresponding local epo-space is not an epo-space.
\end{example}

\subsection{The category of local epo-spaces}

%\tmtextbf{Notation. }\tmtextit{

\begin{notation}
  The following categories are of
particular interest:
\begin{itemizeminus}
  \item $\mathbbm{P}$, the category of epo-spaces and dimaps;
  
  \item $\mathbbm{L}$, the category of local epo-spaces and dimaps;
  
  \item $\tmmathbf{\tmop{Man}}$, the category of topological manifolds with
  boundary and continuous maps.
\end{itemizeminus}
\end{notation}

%}{\hspace*{\fill}}{\medskip}

\begin{proposition}
  There is an adjunction $F \dashv U : \tmmathbf{\tmop{Man}} \longrightarrow
  \mathbb{L}$. The forgetful functor $U : \mathbb{L} \longrightarrow
  \tmmathbf{\tmop{Man}}$ preserves and creates subobjects and limits.
\end{proposition}

\begin{proof}
  Consider a local epo-space $\left( X, \left( U_i \right) \right)$ and the
  inclusion map of topological manifolds with boundary $i : X' \rightarrowtail
  X$. Then $(X' \cap U_i)_{i \in I}$ gives is an atlas on $X'$ with
  respect to the subspace topology and $i$ is a dimap
  \[ i : (X', (X' \cap U_i)) \rightarrowtail (X, \left( U_i \right)) \]
  The converse statement is trivial. In particular, $U$ preserves and creates
  equalizers. As for products, let $\left( X^t, \left( U_i^t \right)_{i \in I
  (t)} \right)_{t \in T}$ be a family of local epo-spaces. Then

  \[ \left( \prod U^t_i  \right)_{t \in T, i \in I (t)} \]

  is an atlas on $\prod_{t \in T} X^t$ and
  \[ \left( \prod_{t \in T} X^t, \left( \prod U^t_i  \right)_{t \in T, i \in I
     (t)} \right) \]

  is a product. The converse statement is again trivial.
\end{proof}

Sums exist in $\mathbb{L}$ and are calculated the usual way. On the other
hand, coequalizers are somehow elusive. We do not know if $\mathbb{L}$ admits
them, yet if it should be the case then they are not created by $U$ for the
following reason. Suppose $(X, \left( U_i \right)) \in \mathbb{L}$ is a local epo-space and let $X'
\subseteq X$ be a subspace of the underlying topological space $X$. It is certainly the case that $\left( U_i / U_i
\cap X' \right)_{i \in I}$ is an open covering $X / X'$ with respect to the
quotient topology. However, quotients of partial orders are preorders which
are not necessarily antisymmetric. Consider for instance $\Delta_D$ as in
example \ref{exa:1} and $\{0, 1\}$ equipped with the discrete order. Passing to the quotient we get
\[ \Delta_D /\{0, 1\} \cong S^1 \]
as topological spaces, yet the order relation becomes a preorder which is not
an order. Nonetheless, an important type of
colimits do exist in $\mathbb{L}$.

\begin{proposition}
  \label{prop:lpc-colim}Let $\left( X, \left( U_i \right) \right) \in
  \mathbb{L}$. The family 
\[\left\{ \left( u_i : U_i \rightarrowtail X
  \right)_{i \in I}, \left( u_{i j} : U_{i j} \rightarrowtail X \right)_{(i,
  j) \in I \times I} \right\}\]
 of dimaps in

%   \begin{center}
%     \epsfig{file=rev2-2.eps}
%   \end{center}

\begin{center}
$
\xy
\ar@{>->}_{\pi_i^{ij}} (25,0)*++{U_{ij}} ; (10,-15)*+{U_i}
\ar@{>->}^{\pi_j^{ij}} (25,0)*++{U_{ij}} ; (40,-15)*+{U_j}
\ar@{>->}_{u_i} (10,-15)*++{U_i} ; (25,-30)*+{X}
\ar@{>->}^{u_j} (40,-15)*++{U_j} ; (25,-30)*+{X}
\ar@{-} (25,-6.45)*{} ; (23,-5)
\ar@{-} (25,-6.45)*{} ; (26.5,-5)
\POS(3,-15)*{\cdots}
\POS(47,-15)*{\cdots}
\endxy
$
\end{center}

  is a colimit of the diagram $\left( \pi_i^{i j} : U_{i j}
  \rightarrow U_i, \pi_j^{i j} : U_{i j} \rightarrow U_j \right)_{(i, j) \in I
  \times I}$.
\end{proposition}

\begin{proof}
  The family $\left\{ \left( u_i : U_i \rightarrowtail X \right)_{i \in I},
  \left( u_{i j} : U_{i j} \rightarrowtail X \right)_{(i, j) \in I \times I}
  \right\}$ is a colimit of the diagram $\left( \pi_i^{i j} : U_{i j}
  \rightarrow U_i, \pi_j^{i j} : U_{i j} \rightarrow U_j \right)_{(i, j) \in I
  \times I}$ in $\tmmathbf{\tmop{Man}}$. Suppose $\left( Y,
  \left( V_j \right) \right) \in \mathbb{L}$ and let $\left( t_i : U_i
  \longrightarrow Y \right)_{i \in I}$ be a family of dimaps such that
  \[ t_i \circ \pi^{i j}_i = t_j \circ \pi^{i j}_j  \overset{\tmop{def} .}{=}
     t_{i j}  \]
  There is the canonical map $c : X \longrightarrow Y$

%   \begin{center}
%     \epsfig{file=rev2-3.eps}
%   \end{center}
  
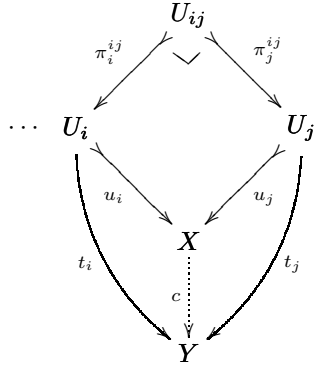
  \begin{center}

$
\xy
\ar@{>->}_{\pi_i^{ij}} (25,0)*++{U_{ij}} ; (10,-15)*+{U_i}
\ar@{>->}^{\pi_j^{ij}} (25,0)*++{U_{ij}} ; (40,-15)*+{U_j}
\ar@{>->}_{u_i} (10,-15)*++{U_i} ; (25,-30)*+{X}
\ar@{>->}^{u_j} (40,-15)*++{U_j} ; (25,-30)*+{X}
\ar@{-} (25,-6.45)*{} ; (23,-5)
\ar@{-} (25,-6.45)*{} ; (26.5,-5)
\ar_{t_i} @/_3ex/ (10,-15)*++{U_i} ; (25,-45)*+{Y}
\ar^{t_j} @/^3ex/ (40,-15)*++{U_j} ; (25,-45)*+{Y}
\ar@{.>}_c (25,-30)*+{X} ; (25,-45)*+{Y}
\POS(3,-15)*{\cdots}
\endxy
$
\end{center}
  
  in $\tmmathbf{\tmop{Man}}$. This map is locally non-decreasing
  since
  \[ c_{|_{c^{- 1} (V_j) \cap U_i}} = (c \circ u_i)_{|_{(c \circ u_i)^{- 1}
     (V_j) \cap U_i}} = t_i{ _{|_{t_i^{- 1} (V_j) \cap U_i}}} \]
  for all $(i, j) \in I \times I$. 
\end{proof}

\begin{remark}
  \label{rem:colim-coeq}More succinctly, $X$ can be calculated as the coequalizer
  \[ \coprod_{i, j \in I^2} U_{\tmop{ij}} \rightrightarrows \coprod_{i \in I}
     U_i \twoheadrightarrow X \]
\end{remark}

\begin{remark} \label{rem:f-comp}
  Let $f : \left( X, \left( U_i \right) \right) \longrightarrow \left( Y,
  \left( W_k \right) \right)$ be a dimap. For each $i \in I$ there is a
  commuting triangle

%   \begin{center}
%     \epsfig{file=rev2-4.eps}
%   \end{center}
  
  \begin{center}
$
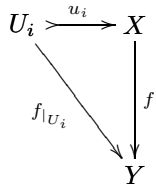
\xy
\ar@{>->}^{u_i} (0,0)*+++{U_i} ; (15,0)*+{X}
\ar^{f} (15,0)*+{X} ; (15,-20)*+{Y}
\ar_{f_{|_{U_i}}} (0,0)*+{U_i} ; (15,-20)*+{Y}
\endxy
$
\end{center}
  
  in $\mathbb{L}$. Hence
  \[ \text{$f = \left[ f_{|_{U_i}} \right]_{i \in I}$} \]
  is the comparison morphism.
\end{remark}

\section{\label{sec:sheaves}Locally ordered spaces as
sheaves\label{sec:losheaves}}

In this section, we exhibit $\mathbb{L}$ as a subcategory of a topos of
sheaves by appropriately restricting a family of Yoneda embeddings.

\subsection{The open-dicover topology}

\begin{remark}
  The assignment
  \[ \tau (X) \overset{\tmop{def} .}{=}  \left\{ \left( X_i \right)_{\in I} |
     \forall i \in I . X_i \subseteq X \tmop{and} \bigcup_{i \in I} X_i = X
     \right\} \]
  determines (a basis of) a Grothendieck topology on $\mathbb{P}$, called the
  {\tmem{open-dicover topology}} (see {\cite{buwo}}). 
\end{remark}

\begin{proposition}
  The site $\left( \mathbb{P}, \tau \right)$ is subcanonical.
\end{proposition}

\begin{proof}
  Assume $\left( U_i \right)_{i \in I}$ is a family of epo-spaces with
  $U_i \subseteq U$ for each $i \in I$ and
  \[ \bigcup_{i \in I} U_i = U \]
  
 Consider a representable presheaf $\mathbb{P} (-, V)$ for some $V
  \in \mathbb{P}$. A matching family for this presheaf with respect to the
  covering family $\left( U_i \right)$ amounts to a family
  \[ \left( f_i : U_i \longrightarrow V \right)_{i \in I} \in \left( \prod_{i
     \in I} \mathbbm{P}(U_i, V) \right) \]
  of continuous non-decreasing functions such that $f_{i|U_i \cap U_j} =
  f_{j|U_i \cap U_j}$ for all $(i, j) \in I^2$. The underlying continuous functions can in this
  case be patched together into a unique continuous function $f : U
  \longrightarrow V$ such that $f_{|U_i} = f_i$ for each $i \in I$. Since the order on the $U_i$'s is the one inherited from $U$, $f$
  is non--decreasing.
\end{proof}

\subsection{An embedding in the topos of sheaves}

\begin{definition}
  The functor $h_{\mathbb{P} }$ is given by
  \[ \begin{array}{cccc}
       h_{\mathbb{P} } : & \mathbb{L} & \longrightarrow & \widehat{
       \mathbb{P}}\\
       & X & \longmapsto & \mathbb{L} \left( -, X \right) :
       \mathbb{P}^{\tmop{op}} \rightarrow \tmmathbf{\tmop{Set}}\\
       & f : X \rightarrow Y & \longmapsto & f^{\ast} = f \circ (-)
     \end{array} \]
\end{definition}

\begin{remark}
  \label{rem:h-fun}The functor $h_{\mathbb{P} }$ verifies
  \[ h_{\mathbb{P} } = y_{\mathbb{L} |_{\mathbb{P}}} \]
  and
  \[ h_{\mathbb{P} |_{\mathbb{P}}} = y_{\mathbb{P}} \]
\end{remark}

\begin{lemma}
  \label{lem:mono-mono}Let $f : A \rightarrow B$ be a dimap in $\mathbb{P}$
  such that $h_{\mathbb{P} } \left( f \right)$ is a mono. Then $f$ itself is a
  mono.
\end{lemma}

\begin{proof}
  Suppose $f (x) = f (y)$. Let $\mathbf{1}$ be the one-point epo-space and let
  $\left\lceil x \right\rceil, \left\lceil y \right\rceil : \mathbf{1}
  \rightarrow A$ be the dimaps choosing $x$ and $y$, respectively. We have
  \[ \text{$h_{\mathbb{P} } (f) \left(  \left\lceil x \right\rceil \right) =
     f^{\ast} \left(  \left\lceil x \right\rceil \right) = f^{\ast} \left( 
     \left\lceil y \right\rceil \right) = h_{\mathbb{P} } (f) \left( 
     \left\lceil y \right\rceil \right)$} \]
  Now $h_{\mathbb{P} } (f)$ is a mono of sheaves, since $\left(
  \mathbb{P}, \tau \right)$ is subcanonical and $h_{\mathbb{P} |_{\mathbb{P}}}
  = y_{\mathbb{P}}$. Hence there is a cover $(A_i \rightarrowtail
  \tmmathbf{1})$ of $\tmmathbf{1} \in \mathbb{P}$ for which $\left\lceil x
  \right\rceil = \left\lceil y \right\rceil$ locally. But the only possible
  covers of $\mathbf{1}$ are identities. Hence $x = y$.
\end{proof}

\begin{remark}
  Notice that the proof of lemma \ref{lem:mono-mono} does not work with an
  arbitrary Grothendieck topology. 
\end{remark}

\begin{lemma}
  \label{lem:h-faith}$h_{\mathbb{P} }$ is faithful.
\end{lemma}

\begin{proof}
Suppose $f, g : \left(
  X, \left( U_i \right) \right) \longrightarrow \left( Y, \left( V_j \right)
  \right)$ are dimaps such that $h_{\mathbb{P} } (f) = h_{\mathbb{P} } (g)$.
  Then
  \begin{eqnarray*}
    f & = & \left[ f_{|_{U_i}} \right]_{i \in I} \\
    & = & \left[ f \circ u_i \right]_{i \in I}\\
    & = & \left[ g \circ u_i \right]_{i \in I} \\
    & = & \left[ g_{|_{U_i}} \right]_{i \in I} \\
    & = & g
  \end{eqnarray*}
\end{proof}

\begin{lemma}
  \label{lem:h-full}$h_{\mathbb{P} }$ is full.
\end{lemma}

\begin{proof}
  Let $\alpha : h_{\mathbb{P} } \left( X \right) \Longrightarrow h_{\mathbb{P}
  } \left( Y \right)$ be a natural transformation, so in particular
  
%   \begin{center}
%     \epsfig{file=rev2-6.eps}
%   \end{center}
  
  \begin{center}
$
\xymatrix{
\mathbb{L}(U_i,X) \ar[r]^{\alpha_{U_i}} 
\ar[d]_{(-)\circ \pi_i^{ij}}
&
\mathbb{L}(U_i,Y) \ar[d]^{(-)\circ \pi_i^{ij}}
\\
\mathbb{L}(U_{ij},X) \ar[r]^{\alpha_{U_{ij}}}
&
\mathbb{L}(U_{ij},Y)
\\
\mathbb{L}(U_j,X) \ar[r]^{\alpha_{U_j}}
\ar[u]^{(-)\circ \pi_j^{ij}}
& \mathbb{L}(U_j,Y) \ar[u]_{(-)\circ \pi_j^{ij}}
}
$
\end{center}
  
  commutes. Since $u_i \circ \pi_i^{i j} = u_j \circ \pi_j^{i j}$
  by construction, it follows that $\alpha_{U_i} (u_i) \circ \pi_i^{i j} =
  \alpha_{U_j} (u_j) \circ \pi_j^{i j}$. By proposition \ref{prop:lpc-colim}
  there is the comparison map $t : X \longrightarrow Y$
  
%   \begin{center}
%     \epsfig{file=rev2-7.eps}
%   \end{center}
  
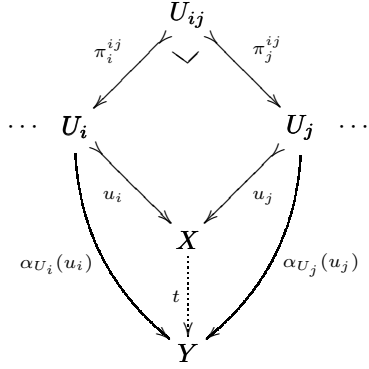
  \begin{center}
$
\xy
\ar@{>->}_{\pi_i^{ij}} (25,0)*++{U_{ij}} ; (10,-15)*+{U_i}
\ar@{>->}^{\pi_j^{ij}} (25,0)*++{U_{ij}} ; (40,-15)*+{U_j}
\ar@{>->}_{u_i} (10,-15)*++{U_i} ; (25,-30)*+{X}
\ar@{>->}^{u_j} (40,-15)*++{U_j} ; (25,-30)*+{X}
\ar@{-} (25,-6.45)*{} ; (23,-5)
\ar@{-} (25,-6.45)*{} ; (26.5,-5)
\ar_{\alpha_{U_i}(u_i)} @/_3ex/ (10,-15)*++{U_i} ; (25,-45)*+{Y}
\ar^{\alpha_{U_j}(u_j)} @/^3ex/ (40,-15)*++{U_j} ; (25,-45)*+{Y}
\ar@{.>}_t (25,-30)*+{X} ; (25,-45)*+{Y}
\POS(3,-15)*{\cdots}
\POS(47,-15)*{\cdots}
\endxy
$
\end{center}
  
  We claim that $\alpha_A = t \circ (-)$ for all $A \in
  \mathbb{P}$. Given $f : A \longrightarrow X$, there is the pullback square

%   \begin{center}
%     \epsfig{file=rev2-8.eps}
%   \end{center}
  
\begin{center}
$
\xy
{\ar@{>->}^{f_i} (0,0)*+++{f^{-1}(U_i)} ; (25,0)*+{U_i}};
{\ar@{>->}_{f^*u_i} (0,0)*+++{f^{-1}(U_i)} ; (0,-20)*+{A}};
{\ar@{>->}_{f} (0,-20)*+++{A} ; (25,-20)*+{X}};
{\ar@{>->}^{u_i} (25,0)*+++{U_i} ; (25,-20)*+{X}};
{\ar@{-} (6,-6)*{} ; (6,-3.5)};
{\ar@{-} (6,-6)*{} ; (3.5,-6)};
(13,-11)*{(*)};
\endxy
$
\end{center}

  for each $i \in I$. In particular, $\left( f^{- 1} (U_i)
  \right)_{i \in I}$ is an atlas on $A$, hence
  \[ \text{$\alpha_A (f) = \left[ \alpha_A (f)_{|_{f^{- 1} (U_i)}} \right]_{i
     \in I}$} \]
  is given by universal property. On the other hand
  \begin{eqnarray*}
    \alpha_A (f)_{|_{f^{- 1} (U_i)}} = & \alpha_A (f) \circ f^{\ast} u_i & \\
    = & \alpha_{f^{- 1} (U_i)} (f \circ f^{\ast} u_i) & \\
    = & \alpha_{f^{- 1} (U_i)} (u_i \circ f_i) & (\ast) \textrm{  commutes}\\
    = & \alpha_{U_i} (u_i) \circ f_i & \\
    = & (t \circ u_i) \circ f_i & \\
    = & t \circ \left( f_{|_{f^{- 1} (U_i)}} \right) & 
  \end{eqnarray*}
  hence
  \begin{eqnarray*}
    \alpha_A (f) & = & \left[ \alpha_A (f)_{|_{f^{- 1} (U_i)}} \right]_{i \in
    I} \\
    & = & \left[ t \circ \left( f_{|_{f^{- 1} (U_i)}} \right) \right]_{i \in
    I}\\
    & = & t \circ \left[ f_{|_{f^{- 1} (U_i)}} \right]_{i \in I}\\
    & = & t \circ f
  \end{eqnarray*}
\end{proof}

\begin{lemma}
  \label{lem:lim-pres}The functor $h_{\mathbb{P} }$preserves limits and
  subobjects.
\end{lemma}

\begin{lemma}
  \label{lem:hx-sheaf}Let $\left( X, \left( U_i \right) \right) \in
  \mathbb{L}$. Then $h_{\mathbb{P} } \left( X, \left( U_i \right) \right) :
  \mathbb{P}^{\tmop{op}} \longrightarrow \tmmathbf{\tmop{Set}}$ is a sheaf
  with respect to the open-dicover topology $\tau$.
\end{lemma}

\begin{proof}
  Let $C \in \mathbb{P}$. Assume that $\left( C_s \right)_{s \in S} \in \tau
  (C)$ and let
  \[ \text{$\left( k_s \in h_{\mathbb{P} } \left( X, \left( U_i \right)
     \right) (C) \right)_{s \in S} = \left( k_s \in \mathbb{L} (C_s, X)
     \right)_{s \in S}$} \]
  be a matching family. By definition of a matching family we have
  \[ k_s \circ \pi_s^{s t} = k_t \circ \pi_t^{s t} \]
  for each pair of indices $\left( s, t \right) \in S \times S$. This family
  has a unique amalgamation $k : C \longrightarrow X$ by proposition
  \ref{prop:lpc-colim}.
\end{proof}

\begin{theorem}
  The functor $h_{\mathbb{P} }$ is fully faithful and preserves limits as well
  as subobjects. Moreover, $h_{\mathbb{P} } (X)$ is a sheaf for all $X \in
  \mathbb{L}$.
\end{theorem}

\begin{proof}
  By remark \ref{rem:h-fun} and lemmata \ref{lem:h-faith}, \ref{lem:h-full},
  \ref{lem:lim-pres} and \ref{lem:hx-sheaf}.
\end{proof}

\subsection{A characterisation of the embedding's image}

\begin{remark}
  \label{rem:unions}In a well-powered regular category, the union
  \[ x' \vee x'' : X' \vee X'' \rightarrowtail X \]
  of two subobjects $x' : X' \rightarrowtail X$ and $x'' : X'' \rightarrowtail
  X$ of $X$ can be calculated as the comparison morphism

%   \begin{center}
%     \epsfig{file=rev2-9.eps}
%   \end{center}
  
  \begin{center}
$
\xy
{\ar@{>->} (0,0)*+++{X'\wedge X''} ; (35,0)*+{X''}};
{\ar@{>->} (0,0)*+++{X'\wedge X''} ; (0,-30)*+{X'}};
{\ar@{>->}_{x'} (0,-30)*+++{X'} ; (35,-30)*+{X}};
{\ar@{>->}^{x''} (35,0)*+++{X''} ; (35,-30)*+{X}};
{\ar@{>.>}_{x' \vee x''} (22.5,-18)*+++{X' \vee X''} ; (35,-30)*+{X}};
{\ar@{>->} (35,0)*+++{X''} ; (22.5,-18)*++{X' \vee X''}};
{\ar@{>->} (0,-30)*+++{X'} ; (22.5,-18)*++{X' \vee X''}};
{\ar@{-} (5,-5)*{} ; (5,-2.5)};
{\ar@{-} (5,-5)*{} ; (2.5,-5)};
{\ar@{-} (16,-12.5)*{} ; (18.5,-12.5)};
{\ar@{-} (16,-12.5)*{} ; (16,-14.5)};
\endxy
$
\end{center}
  
  from the inscribed pushout object. This remains true for
  set-indexed unions if the category is complete. In this case, set-indexed
  unions can be calculated as the colimit of the diagram given by binary
  intersections:

%   \begin{center}
%     \epsfig{file=rev2-10.eps}
%   \end{center}
  
  \begin{center}
$
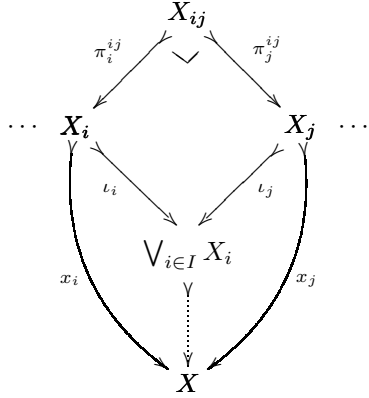
\xy
\ar@{>->}_{\pi_i^{ij}} (25,0)*++{X_{ij}} ; (10,-15)*+{X_i}
\ar@{>->}^{\pi_j^{ij}} (25,0)*++{X_{ij}} ; (40,-15)*+{X_j}
\ar@{>->}_{\iota_i} (10,-15)*++{X_i} ; (23.5,-28)*+{}
\ar@{>->}^{\iota_j} (40,-15)*++{X_j} ; (26.5,-28)*+{}
\ar@/_1.5pc/@{>->}_{x_i} (10,-15)*++{X_i} ; (25,-49)*+{X}
\ar@/^1.5pc/@{>->}^{x_j} (40,-15)*++{X_j} ; (25,-49)*+{X}
\ar@{>.>} (25,-37.5)*{} ; (25,-49)*+{X}
\ar@{-} (25,-6.45)*{} ; (23,-5)
\ar@{-} (25,-6.45)*{} ; (26.5,-5)
\POS(25,-32)*{\bigvee_{i \in I} X_i}
\POS(3,-15)*{\cdots}
\POS(47,-15)*{\cdots}
\endxy
$
\end{center}
  
  In particular, the above is true in any topos since topoi are
  regular, well-powered and complete.
\end{remark}

\begin{definition}
  A dimap is {\tmem{\'etale}} if the underlying continuous map is a local
  homeomorphism.
\end{definition}

\begin{remark}
  \label{rem:etalpull}\'Etale dimaps are stable under pullback.
\end{remark}

\begin{definition}
  A morphism $u : F \longrightarrowlim G$ in $\tmop{Sh} ( \mathbb{P}, \tau)$
  is {\tmem{$\mathbb{P}$-local}} if
  \begin{enumerateroman}
    \item for all $X \in \mathbb{P}$ and morphisms $h_{\mathbb{P}} (X)
    \longrightarrow G$ there is an $Y \in \mathbb{P}$ such that
    \[ F \times_G h_{\mathbb{P}} (X) \cong h_{\mathbb{P}} (Y) \]
    \item given the projection $p : F \times_G h_{\mathbb{P}} (X) \cong
    h_{\mathbb{P}} (Y) \longrightarrow h_{\mathbb{P}} (X)$ from the fibred
    product above, the image $\left( h_{\mathbb{P}_{Y, X}} \right)^{^{- 1}}
    (p_{}$) of $p$ under the inverse of the bijection
    \[ \left. h_{\mathbb{P} _{Y, X}} : \mathbb{P} (Y, X)
       \longrightarrowlim^{\cong} \tmop{Sh} ( \mathbb{P}, \tau)
       (h_{\mathbb{P}} (Y), h_{\mathbb{P}} \left( X \right) \right) \]
    is an \'etale dimap.
  \end{enumerateroman}
\end{definition}

\begin{theorem}
  \label{theo:plocal}Let $L \in \tmop{Sh} ( \mathbb{P}, \tau)$. The following
  are equivalent:
  \begin{enumerateroman}
    \item $L$ is a local epo-space;
    
    \item there is a family
    \[ \left( \kappa_i : h_{\mathbb{P} } (U_i) \longrightarrow L \right)_{i
       \in I} \]
    of $\mathbb{P}$-local monos such that the canonical morphism
    \[ [\kappa_i]_{i \in I} : \coprod_{i \in I} h_{\mathbb{P} } (U_i)
       \longrightarrow L \]
    is an epi.
  \end{enumerateroman}
\end{theorem}

\begin{proof}

  ``$\Longrightarrow$'' Let $L \overset{\tmop{def} .}{=}
  h_{\mathbb{P} } \left( U, \left( U_i \right) \right)$. The canonical
  morphism
  \[ [h_{\mathbb{P} } (u_i)]_{i \in I} : \coprod_{i \in I} h_{\mathbb{P} }
     (U_i) \longrightarrow L = h_{\mathbb{P} } (U) \]
  is a local epi at any object, hence an epi of sheaves. Similarly,
  $h_{\mathbb{P} } (u_i)$ is a local mono at any object and so a mono of
  sheaves, this for all $i \in I$.
  
  We claim that $h_{\mathbb{P} } (u_i)$ is $\mathbb{P}$-local for all $i \in
  I$. Suppose $X \in \mathbb{P}$ and consider the pullback square \
  
%   \begin{center}
%     \epsfig{file=rev2-11.eps}
%   \end{center}
  
  \begin{center}
$
\xy
{\ar@{>->}^{\pi_1} (0,0)*+++{M} ; (25,0)*+{h_\mathbb{P}(X)}};
{\ar@{->}_{\pi_2} (0,0)*+{M} ; (0,-20)*+{h_\mathbb{P}(U_i)}};
{\ar@{>->}_{h_\mathbb{P}(u_i)} (0,-20)*+++{h_\mathbb{P}(U_i)} ; (25,-20)*+{h_\mathbb{P}(U)}};
{\ar@{->}^{\phi} (25,0)*+{h_\mathbb{P}(X)} ; (25,-20)*+{h_\mathbb{P}(U)}};
{\ar@{-} (5,-5)*{} ; (5,-2.5)};
{\ar@{-} (5,-5)*{} ; (2.5,-5)};
\endxy
$
\end{center}
  
  We have $\phi = h_{\mathbb{P} } (f)$ for some dimap $f : X
  \longrightarrow U$ since $h_{\mathbb{P}}$ is full. Hence
  \[ M (P) \cong \{(u, v) \in \mathbb{P} (P, U_i) \times \mathbb{P} (P, X) |
     u_i \circ u = f \circ v\} \cong \mathbb{P} (P, U_i \times_U X) =
     h_{\mathbb{P} } (U_i \times_U X) (P) \]
  for all $P \in \mathbb{P}$, so $M \cong h_{\mathbb{P} } (U_i \times_U X)$.
  Now $\pi_1 = h_{\mathbb{P} } (p_1)$ with $p_1 : U_i \times_U X
  \longrightarrow X$ the corresponding projection from the fibred product in
  $\mathbb{P}$. But $p_1$ is obtained by pulling back $u_i$, which is an
  \'etale dimap, hence $p_1$ is an \'etale dimap by remark \ref{rem:etalpull}.

  ``$\Longleftarrow$'' We proceed here in three steps:
  \begin{enumeratenumeric}
    \item the construction of a local epo-space $U \cong \tmop{colim}_{i \in
    I} U_i$;
    
    \item the proof of the assertion $L \cong \tmop{colim}_{i \in I}
    h_{\mathbb{P} }  \left( U_i \right)$;
    
    \item the proof of the assertion $h_{\mathbb{P} }  \left( \tmop{colim}_{i
    \in I} U_i  \right) \cong \tmop{colim}_{i \in I} h_{\mathbb{P} }  \left(
    U_i \right)$.
  \end{enumeratenumeric}
  Step 1. Let $i, j \in I$. By hypothesis there is an $U_{i j} \in \mathbb{P}$
  along with the \'etale dimaps $p_{i j} : U_{i j} \longrightarrow U_i$ and
  $q_{i j} : U_{i j} \longrightarrow U_j$ assembling to the pullback square

%   \begin{center}
%     \epsfig{file=rev2-12.eps}
%   \end{center}
  
  \begin{center}
$
\xy
{\ar@{>->}^{p_{ij}^\ast} (0,0)*+++{h_\mathbb{P}(U_{ij})} ; (25,0)*+{h_\mathbb{P}(U_i)}};
{\ar@{>->}_{q_{ij}^\ast} (0,0)*+++{h_\mathbb{P}(U_{ij})} ; (0,-20)*+{h_\mathbb{P}(U_j)}};
{\ar@{>->}_{\kappa_j} (0,-20)*+++{h_\mathbb{P}(U_j)} ; (25,-20)*+{L}};
{\ar@{>->}^{\kappa_i} (25,0)*+++{h_\mathbb{P}(U_i)} ; (25,-20)*+{L}};
{\ar@{-} (6,-6)*{} ; (6,-3.5)};
{\ar@{-} (6,-6)*{} ; (3.5,-6)};
\endxy
$
\end{center}
  
  in $\tmop{Sh} ( \mathbb{P}, \tau)$. Doing the construction for
  all pairs of indices $(i, j) \in I^2$ yields a family $\left( U_{i j}
  \right)_{(i, j) \in I^2}$ of epo-spaces. \ The $p_{i j}$'s and the $q_{_{i
  j}}$'s are \'etale since the $\kappa_i$'s and the $\kappa_j$'s are
  $\mathbb{P}$-local by hypothesis. Moreover, they are monos by lemma \
  \ref{lem:mono-mono}. Hence $U_{i j} \rightarrowtail U_i, U_j$ represent open
  subobjects.
  
  Let \ $U \cong \tmop{colim}_{i \in I} U_i$ be the colimit of the diagram

%   \begin{center}
%     \epsfig{file=rev2-13.eps}
%   \end{center}
  
  \begin{center}
$
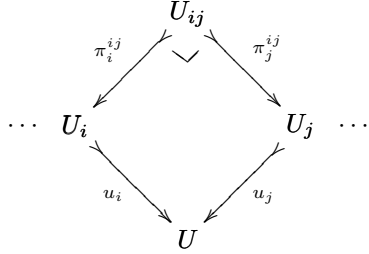
\xy
\ar@{>->}_{\pi_i^{ij}} (25,0)*++{U_{ij}} ; (10,-15)*+{U_i}
\ar@{>->}^{\pi_j^{ij}} (25,0)*++{U_{ij}} ; (40,-15)*+{U_j}
\ar@{>->}_{u_i} (10,-15)*++{U_i} ; (25,-30)*+{U}
\ar@{>->}^{u_j} (40,-15)*++{U_j} ; (25,-30)*+{U}
\ar@{-} (25,-6.45)*{} ; (23,-5)
\ar@{-} (25,-6.45)*{} ; (26.5,-5)
\POS(3,-15)*{\cdots}
\POS(47,-15)*{\cdots}
\endxy
$
\end{center}
  
  The $u_i$'s are monos by construction. Recall from remark
  \ref{rem:colim-coeq} that $U$ can be calculated as the quotient

%   \begin{center}
%     \epsfig{file=rev2-14.eps}
%   \end{center}
  
  \begin{center}
$
\xymatrix{ 
\coprod  U_{i\:j}
\ar@<3.5pt>[rr]^{[in_i \circ p_{i j}]}
\ar@<-3.5pt>[rr]_{[in_j \circ q_{i j}]} 
& & \coprod U_i \ar@<0pt>@{->>}[r]^q
& U}
$
\end{center}
  
In particular,
  $U_i$ is open for all $i \in I$ since
  \[ q^{- 1} (U_i) = \coprod_{j \in I} U_{i j} \subseteq \coprod_{i \in I} U_i
  \]
  and the $U_{i j}$'s are open in $U$.
  
  Finally, we need to show that $U_{i j} \cong U_i \times_U U_j$. Consider

 %  \begin{center}
%     \epsfig{file=rev2-15.eps}
%   \end{center}
  
 \begin{center}
$
\xy
{\ar@/^1.5pc/@{>->}^{p_{i j}} (-14,14)*+++{U_{i j}} ; (25,0)*+{U_j}};
{\ar@/_1.5pc/@{>->}_{p_{i j}} (-14,14)*+++{U_{i j}} ; (0,-20)*+{U_i}};
{\ar@{>.>}^{t} (-14,14)*++{U_{i j}} ;(0,0)*++{U_i \times_U U_j} };
{\ar@{>->}^{\pi_j} (0,0)*+++{U_i \times_U U_j} ; (25,0)*+{U_j}};
{\ar@{>->}_{\pi_i} (0,0)*+++{U_i \times_U U_j} ; (0,-20)*+{U_i}};
{\ar@{>->}_{u_i} (0,-20)*+++{U_i} ; (25,-20)*+{U}};
{\ar@{>->}^{u_j} (25,0)*+++{U_j} ; (25,-20)*+{U}};
{\ar@{-} (6,-6)*{} ; (6,-3.5)};
{\ar@{-} (6,-6)*{} ; (3.5,-6)};
\endxy
$
\end{center} 
  
  The \'etale dimap $t = \left\langle p_{i j}, q_{i j}
  \right\rangle_U$ is surjective by construction of $U$. It is also injective
  since $p_{i j}$ is a mono. But an \'etale bijection is a homeomorphism.
  Hence $\left( U_i \right)_{i \in I}$ is an atlas on $U$.
  
  Step 2. Consider

%   \begin{center}
%     \epsfig{file=rev2-16.eps}
%   \end{center}
  
  \begin{center}
$
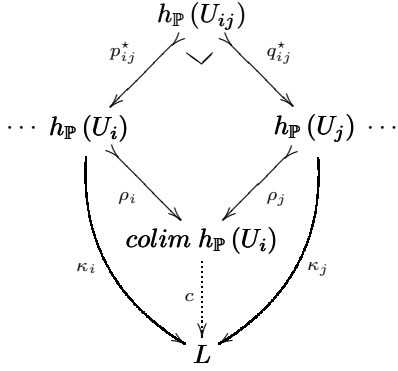
\xy
\ar@{>->}_{p_{ij}^\star} (25,0)*++{h_{\mathbb P}\left ( U_{ij} \right )} ; (10,-15)*+{h_{\mathbb P}\left ( U_i \right )}
\ar@{>->}^{q_{ij}^\star} (25,0)*++{ h_{\mathbb P}\left ( U_{ij} \right )} ; (40,-15)*+{ h_{\mathbb P}\left ( U_j \right )}
\ar@{>->}_{\rho_i} (10,-15)*++{ h_{\mathbb P}\left ( U_i \right )} ; (25,-30)*+{colim \; h_{\mathbb P}\left ( U_i \right )}
\ar@{>->}^{\rho_j} (40,-15)*++{ h_{\mathbb P}\left ( U_j \right )} ; (25,-30)*+{colim \; h_{\mathbb P}\left ( U_i \right )}
\ar@{-} (25,-6.45)*{} ; (23,-5)
\ar@{-} (25,-6.45)*{} ; (26.5,-5)
\ar_{\kappa_i} @/_4ex/ (10,-15)*++{h_{\mathbb P}\left (U_i\right )} ; (25,-45)*+{L}
\ar^{\kappa_j} @/^4ex/ (40,-15)*++{h_{\mathbb P}\left (U_j\right )} ; (25,-45)*+{L}
\ar@{.>}_c (25,-30)*+{colim \;h_{\mathbb P}\left ( U_i \right )} ; (25,-45)*+{L}
\POS(1,-15)*{\cdots}
\POS(49,-15)*{\cdots}
\endxy
$
\end{center}
  
  with $c$ the canonical morphism. Then

 %  \begin{center}
%     \epsfig{file=rev2-17.eps}
%   \end{center}
  
  \begin{center}
$
\xymatrix{
\coprod_{i\in I} h_{\mathbb P}\left ( U_i \right ) 
\ar@{->>}[rr]^{\left [ \kappa_i\right ]_{i \in I}} 
\ar[dr]_{\left [ \rho_i\right ]_{i \in I}}
& & L
\\
& colim \; h_{\mathbb P}\left ( U_i \right )
\ar[ur]_c &
}
$
\end{center}
  
  commutes so $c$ is an epi. But $c$ is also a mono, being the
  inclusion of a set-indexed union of subobjects (c.f. remark
  \ref{rem:unions}). Hence $c$ is an iso since a topos is balanced.
  
  Step 3. Consider

%   \begin{center}
%     \epsfig{file=rev2-18.eps}
%   \end{center}
  
  \begin{center}
$
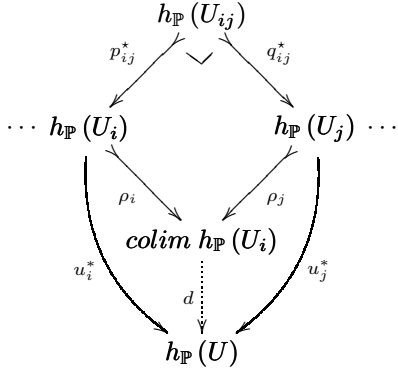
\xy
\ar@{>->}_{p_{ij}^\star} (25,0)*++{h_{\mathbb P}\left ( U_{ij} \right )} ; (10,-15)*+{h_{\mathbb P}\left ( U_i \right )}
\ar@{>->}^{q_{ij}^\star} (25,0)*++{ h_{\mathbb P}\left ( U_{ij} \right )} ; (40,-15)*+{ h_{\mathbb P}\left ( U_j \right )}
\ar@{>->}_{\rho_i} (10,-15)*++{ h_{\mathbb P}\left ( U_i \right )} ; (25,-30)*+{colim \; h_{\mathbb P}\left ( U_i \right )}
\ar@{>->}^{\rho_j} (40,-15)*++{ h_{\mathbb P}\left ( U_j \right )} ; (25,-30)*+{colim \; h_{\mathbb P}\left ( U_i \right )}
\ar@{-} (25,-6.45)*{} ; (23,-5)
\ar@{-} (25,-6.45)*{} ; (26.5,-5)
\ar_{u^*_i} @/_4ex/ (10,-15)*++{h_{\mathbb P}\left (U_i\right )} ; (25,-45)*+{h_{\mathbb P}\left ( U \right)}
\ar^{u^*_j} @/^4ex/ (40,-15)*++{h_{\mathbb P}\left (U_j\right )} ; (25,-45)*+{h_{\mathbb P}\left ( U \right)}
\ar@{.>}_d (25,-30)*+{colim \;h_{\mathbb P}\left ( U_i \right )} ; (25,-45)*+{h_{\mathbb P}\left ( U \right)}
\POS(1,-15)*{\cdots}
\POS(49,-15)*{\cdots}
\endxy
$
\end{center}
  
  The canonical morphism $d$ is a representative of the inclusion
  of the union of subobjects

  \[ \bigvee_{i \in I} h_{\mathbb{P} } (U_i) \cong \tmop{colim}_{i \in I}
     h_{\mathbb{P} } (U_i) \]

  In particular, $d$ is a mono.
  
  It is also the case that $d$ is an epi. To see this, let $A \in \mathbb{P}$
  and let
  \[ q^{\ast} : \mathbb{L} \left( A, \coprod U_i \right) \longrightarrow
     \mathbb{L} (A, U) \]
  be the component of $h_{\mathbb{P} } (q)$ at $A$. Suppose $f \in \mathbb{L}
  (A, U)$. The assignment
  \[ A_i  \overset{\tmop{def} .}{=} f^{- 1} \left( q \circ \tmop{in}_i \right)
  \]
  determines a cover of $A$ and $q^{\ast}$ is locally surjective at this
  cover. Hence $h_{\mathbb{P} } (q)$ is an epi of sheaves. A similar argument
  shows that
  \[ \left[ \tmop{in}_i^{\ast} \right] : \coprod h_{\mathbb{P} } (U_i)
     \longrightarrow h_{\mathbb{P} } \left( \coprod U_i \right) \]
  is an epi of sheaves as well. Finally, the top row of

%   \begin{center}
%     \epsfig{file=rev2-19.eps}
%   \end{center}
  
  \begin{center}
$
\xymatrix{ 
\coprod  h_{\mathbb P}\left (U_{ij}\right )
\ar@<3.5pt>[rr]^{[in^\prime_i \circ p^*_{i j}]}
\ar@<-3.5pt>[rr]_{[in^\prime_j \circ q^*_{i j}]} 
& & \coprod h_{\mathbb P}\left (U_i \right )\ar@<0pt>@{->>}[r]
\ar@<0pt>@{->>}[d]_{\left [in^*_i \right ]}
& **[r] colim \: h_{\mathbb P}\left (U_i\right )
\ar@{.>}[d]^d
\\
&& h_{\mathbb P}\left ( \coprod U_i \right ) 
\ar@{->>}[r]_{h_{\mathbb P}\left (q\right)}
& h_{\mathbb P}\left ( U \right ) }
$
\end{center}
  
  is a coequalizer diagram. It is easy to see that
  \[ \text{$h_{\mathbb{P} } (q) \circ \left[ \tmop{in}_i^{\ast} \right] \circ
     \left[ \tmop{in}'_i \circ p_{i j}^{\ast} \right] = h_{\mathbb{P} } (q)
     \circ \left[ \tmop{in}_i^{\ast} \right] \circ \left[ \tmop{in}'_i \circ
     q_{i j}^{\ast} \right] $} \]
  and that $d$ is the canonical morphism. In particular, $d$ is an epi since
  it is a second factor of an epi.
\end{proof}

\begin{remark}
  Theorem \ref{theo:plocal} says that the site $( \mathbb{P}, \tau)$ along
  with the class of \'etale dimaps form what is called a {\tmem{geometric
  context}} in {\cite{toen-dea}}.
\end{remark}

\section{Intervals and homotopy theories\label{sec:intervals}}

\subsection{\label{subsec:cellular}Cellular models}A cellular model generates
(in a certain sense) all the monos in a given category.

\begin{definition}
  Let $\mathbb{C}$ be a category and let $M \in \mathbb{C}_1$ be a class of
  morphisms in $\mathbb{C}$. Then $M - \tmop{inj}$ is the class if all
  morphisms having the {\tmem{right lifting property}} with respect to $M$ and
  $M - \tmop{cof}$ is the class if all morphisms having the {\tmem{left
  lifting property}} with respect to $M - \tmop{inj}$.
\end{definition}

\begin{definition}
  A cellular model of a category $\mathbb{C}$ is a set of monos $\mathfrak{M}
  \subset \mathbb{C}_1$ such that $\mathfrak{M}- \tmop{cof}$ is the class of
  all monos in $\mathbb{C}$.
\end{definition}

\begin{proposition}
  \label{prop:cellular-model}Any locally presentable category $\mathbb{C}$
  with effective unions of subobjects and monos closed under transfinite
  composition admits a cellular model $\mathfrak{M} \subset \mathbb{C}_1$.
\end{proposition}

\begin{remark}
  One such $\mathfrak{M}$ is the set of (representatives of) subobjects of
  (representatives of) regular quotients of the set of $\mathbb{C}$'s strong
  generators (see {\cite{beke:2001}} for a proof). In particular, any topos
  verifies the assumptions of proposition \ref{prop:cellular-model} and so
  admits a cellular model. Of course, more practical cellular models are known
  in cases of interest, like the set $\left( \partial \Delta [n]
  \hookrightarrow \Delta [n] \right)_{n \in \mathbb{N}}$ of boundary
  inclusions of simplicial sets.
\end{remark}

\subsection{Intervals}Let $\mathbb{C}$ be a category with coproducts and
pullbacks.

\begin{definition}
  \label{def:cyl}A {\tmem{cylinder}} $\mathcal{I} = (I, \partial^0,
  \partial^1, \sigma)$ on $\mathbb{C}$ is given by the following
  data:{\tmstrong{}}
  \begin{itemizeminus}
    \item an endofunctor $I : \mathbb{C} \longrightarrow \mathbb{C}$;
    
    \item natural transformations $\partial^{_0}, \partial^1 :
    \tmop{id}_{\mathbb{C}} \Rightarrow I$ and $\sigma : I \Rightarrow
    \tmop{id}_{\mathbb{C}}$ such that $\sigma \circ \partial^0 = \sigma \circ
    \partial^1 = \tmop{id}_{\tmop{id}_{\mathbb{C}}}$.
  \end{itemizeminus}
  A {\tmem{morphism of cylinders}} $\iota : \mathcal{I} \longrightarrow
  \mathcal{I}'$ is a natural transformation $\iota : I \longrightarrow I'$
  such that
  \begin{enumerateroman}
    \item $\iota \circ \partial^e_{\mathcal{I}} = \partial^e_{\mathcal{I}'}$
    for $e \in \{0, 1\}$;
    
    \item $\sigma_{\mathcal{I}} = \sigma_{\mathcal{I}'} \circ \iota$.
  \end{enumerateroman}
\end{definition}

\begin{definition}
  Let $\mathcal{I}$ be a cylinder. Morphism $f_0, f_1 : X \longrightarrow Y$
  are $\mathcal{I}$-homotopic if there is a morphism $h : I (X)
  \longrightarrow Y$, called $\mathcal{I}$-{\tmem{homotopy}}, such that $h
  \circ \partial^e = f_e$ for $e \in \{0, 1\}$.
\end{definition}

\begin{notation}
  We write $f \sim_{\mathcal{I}} g$ to indicate that $f$ and $g$ are
  $\mathcal{I}$-homotopic.
\end{notation}

\begin{definition}
  \label{def:homoto-equivalence}Let $\mathcal{I}$ be a cylinder. A morphism $f
  : X \longrightarrow Y$ in $\mathcal{E}$ is an $\mathcal{I}$-{\tmem{homotopy
  equivalence}} if there is a morphism $g : Y \longrightarrow X$ such that $g
  \circ f \sim_{\mathcal{I}} \tmop{id}_X$ and $f \circ g \sim_{\mathcal{I}}
  \tmop{id}_Y$.
\end{definition}

\begin{definition}
  \label{def-homdat}A cylinder is {\tmem{cartesian}} if the naturality square

%   \begin{center}
%     \epsfig{file=rev2-20.eps}
%   \end{center}
  
\begin{center}
$
 \xy
{\ar@{>->}^{j} (0,0)*+++{A} ; (25,0)*+{B}};
{\ar@{>->}_{\partial^e_A} (0,0)*+++{A} ; (0,-20)*+{I(A)}};
{\ar@{>->}_{I(j)} (0,-20)*+++{I(A)} ; (25,-20)*+{I(B)}};
{\ar@{>->}^{\partial^e_B} (25,0)*+++{B} ; (25,-20)*+{I(B)}};
{\ar@{-} (5,-5)*{} ; (5,-2.5)};
{\ar@{-} (5,-5)*{} ; (2.5,-5)};
\endxy
$
\end{center}

  is a pullback square for all monos $j$ and $e \in \{0, 1\}$. An
  {\tmem{interval}} $\mathcal{I}$ is a cartesian cylinder
  \[ \text{ $\mathcal{I} = (I, \partial_0, \partial_1, \sigma)$} \]
  such that
  \begin{enumerateroman}
    \item $I$ preserves monos and colimits;
    
    \item the canonical morphism $[\partial^0_C, \partial^1_C] : C + C
    \longrightarrow I (C)$ is mono for all $C \in \mathbb{C}$.
  \end{enumerateroman}
  A {\tmem{morphism of intervals}} $\iota : \mathcal{I} \longrightarrow
  \mathcal{I}'$ is a morphism of the underlying cylinders.
\end{definition}

\begin{remark}
  \label{rem:homoto-equivalence}Let $\iota : \mathcal{I} \longrightarrow
  \mathcal{I}'$ be a morphism of intervals and $u, v : X \longrightarrow Y$
  two morphisms in $\mathbb{C}$. Then
  \[ u \sim_{\mathcal{I}'} v \Longrightarrow u \sim_{\mathcal{I}} v \]
  since the homotopy $h : I' (X) \longrightarrow Y$ witnessing the antecedent
  extends to a homotopy
  \[ h \circ \iota_X : I (X) \longrightarrow Y \]
  witnessing the conclusion. In particular, an $\mathcal{I}'$-homotopy
  equivalence is always an $\mathcal{I}$-homotopy equivalence.
\end{remark}

\begin{remark}
  Colimits in a topos are universal. It follows that

%   \begin{center}
%     \epsfig{file=rev2-21.eps}
%   \end{center}
  
  \begin{center}
$
\xy
{\ar@{>->}^{j+j} (0,0)*+++{A+A} ; (25,0)*+{B+B}};
{\ar@{>->}_{\left[\partial^0_A,\partial^1_A\right]} (0,0)*+++{A+A} ; (0,-20)*+{I(A)}};
{\ar@{>->}_{I(j)} (0,-20)*+++{I(A)} ; (25,-20)*+{I(B)}};
{\ar@{>->}^{\left[\partial^0_B,\partial^1_B\right]} (25,0)*+++{B+B} ; (25,-20)*+{I(B)}};
{\ar@{-} (5,-5)*{} ; (5,-2.5)};
{\ar@{-} (5,-5)*{} ; (2.5,-5)};
\endxy
$
\end{center}
  
  is a pullback square.
\end{remark}

\subsection{Anodyne extensions and model structures}For the rest of this
section we fix a topos $\mathcal{E}$ and a cellular model $\mathfrak{M}$
thereof.

\begin{definition}
  Let $\mathcal{I}$ be an interval. Given a set $L$ of monos in $\mathcal{E}$,
  let
  \[ \tmop{sat} (L) \overset{\tmop{def} .}{=}  \left\{ I (l) \vee \left[
     \partial^0_{\tmop{cod} (m)}, \partial^1_{\tmop{cod} (m)} \right]  \;\;|\;\; l \in
     L \right\} \]
  Then
  \begin{itemizeminus}
    \item $\Lambda_{\mathcal{I}}^0  \overset{\tmop{def} .}{=}  \left\{ I (m)
    \vee \partial^e_{\tmop{cod} (m)}  \;\;|\;\; m \in \mathfrak{M}, e \in \{0, 1\}
    \right\}$;
    
    \item $\Lambda_{\mathcal{I}}^{n + 1}  \overset{\tmop{def} .}{=} \tmop{sat}
    (\Lambda^n_{\mathcal{I}})$ for $n \geqslant 0$;
    
    \item $\Lambda_{\mathcal{I}} \overset{\tmop{def} .}{=}  \bigcup_{n
    \geqslant 0} \Lambda^n_{\mathcal{I}}$.
  \end{itemizeminus}
  A morphism in $\left( \Lambda_{\mathcal{I}} \right) - \tmop{inj}$ is called
  an {\tmem{$\mathcal{I}$-naive fibration}}. An object $X \in \mathcal{E}$ is
  $\mathcal{I}$-naively fibrant if the canonical morphism $!_X : X
  \longrightarrowlim 1$ is an {\tmem{$\mathcal{I}$-naive fibration}}. A
  morphism in $\left( \Lambda_{\mathcal{I}} \right) - \tmop{cof}$ is called an
  $\mathcal{I}$-{\tmem{anodyne extension}}. 
\end{definition}

\begin{notation}
  We shall write $\mathcal{A}_{\mathcal{I}}$ for the class of
  $\mathcal{I}$-anodyne extensions and $\widehat{\mathcal{F}}_{\mathcal{I}}$
  for the class of $\mathcal{I}$-naive fibrations and
  $\mathcal{E}^{\tmop{nf}}_{\mathcal{I}}$ for the subcategory of
  $\mathcal{I}$-naively fibrant objects.
\end{notation}

\begin{theorem}
  {\dueto{Cisinski}}\label{theo:cisinski}Let $\text{$\mathcal{I}$}$ be an
  interval. The $\mathcal{I}$-homotopy relation is an equivalence relation on
  $\mathcal{E} (X, T$) provided $T$ is $\mathcal{I}$-naively fibrant.
  $\mathcal{E}$ admits a cofibrantly generated model structure for which the
  cofibrations are the monos and the weak equivalences are the morphisms $f :
  X \longrightarrow Y$ inducing a bijection
  \[ f^{\ast} : \mathcal{E} (Y, T) /_{\sim} \longrightarrowlim^{\cong} 
     \mathcal{E} (X, T) /_{\sim} \]
  on $\mathcal{I}$-homotopy classes for all $\mathcal{I}$-naively fibrant objects $T
  \in \mathcal{E}$. 
\end{theorem}

It is to be said that Cisinski's original theorem {\cite{cisinski-topos}} is
more general since there is a further parameter allowed, namely an arbitrary
set of monos $\mathfrak{S}$ can be added to $\Lambda_0$. Theorem
\ref{theo:cisinski} above states thus the special case when $\mathfrak{S}=
\varnothing$ (which is enough for our purposes). As pointed out by Jardine
{\cite{cahoth}}, in the general case the same homotopy theory is presented by
Bousfield-localising by $S$ a model structure obtained with the above process
for $\mathfrak{S}= \varnothing$.

Since model structures on topoi constructed following the recipe of theorem
\ref{theo:cisinski} are fully determined by the ``input'' interval
$\mathcal{I}$, we shall call such model structures $\mathcal{I}$-model
structures. Since we will be dealing with different $\mathcal{I}$-model
structures on the same topos $\mathcal{E}$, let us make the convention to
write $\left( \mathcal{E}, \mathcal{I} \right)$ when seeing $\mathcal{E}$ as
an ``$\mathcal{I}$-model category'' with respect to the interval
$\mathcal{I}$.

Next we compile some useful facts about $\mathcal{I}$-model structures.

\begin{notation}
  We shall write $\mathcal{W}_{\mathcal{I}}$ for the class of
  $\mathcal{I}$-weak equivalences and $\mathcal{F}_{\mathcal{I}}$ for the
  class of $\mathcal{I}$- fibrations.
\end{notation}

\begin{remark}
  \label{rem:homoequs1}An $\mathcal{I}$-homotopy equivalence is always an
  $\mathcal{I}$-weak equivalence. 
\end{remark}

\begin{proposition}
  \label{prop:cisinski1}In an $\mathcal{I}$-model structure:
  \begin{enumerateroman}
    \item $X \in \mathcal{E}$ is $\mathcal{I}$-fibrant if and only if it is
    $\mathcal{I}$-naively fibrant;
    
    \item $\Lambda_{\mathcal{I}} \subseteq \mathcal{A}_{\mathcal{I}} \subseteq
    \mathcal{C}_{\mathcal{I}} \cap \mathcal{W}_{\mathcal{I}}$;
    
    \item $\partial_X^e \in \mathcal{A}_{\mathcal{I}}$ for $e \in \{0, 1\}$
    and all $X \in \mathcal{E}$.
  \end{enumerateroman}
\end{proposition}

\begin{proof}
  By propositions {\cite[2.20]{cisinski-topos}},
  {\cite[2.23]{cisinski-topos}}, {\cite[2.28]{cisinski-topos}} and remark
  $\ref{rem:unions}$.
\end{proof}

\begin{remark}
  \label{rem:homoequs2}Suppose $w : X \longrightarrow Y$ is an
  $\mathcal{I}$-weak equivalence and $f \sim_{\mathcal{I}} w$. Then $f$ is an
  $\mathcal{I}$-weak equivalence by proposition \ref{prop:cisinski1}. It is
  further the case that $\sigma_X$ (see definition \ref{def:cyl}) is an
  $\mathcal{I}$-weak equivalence for each object $X \in \mathcal{E}$.
\end{remark}

\subsection{Quillen pairs induced by morphisms intervals}

\begin{lemma}
  \label{lem:weaki1}Suppose there is a morphism of intervals $\mathcal{I}
  \longrightarrow \mathcal{I}'$ which is an $\mathcal{I}$-weak equivalence
  component-wise. Then
  \[ \Lambda_{\mathcal{I}'} \subseteq \mathcal{C}_{\mathcal{I}} \cap
     \mathcal{W}_{\mathcal{I}} \]
\end{lemma}

\begin{proof}
  Let $\iota : \mathcal{I} \longrightarrow \mathcal{I}'$ be a morphism of
  intervals which is a section-wise $\mathcal{I}$-weak equivalence. Let $j : A
  \rightarrowtail B$ be in $\mathfrak{M}$ and $e \in \{0, 1\}$. We have
  $\iota_A \circ \partial^e_{\mathcal{I}, A} = \partial^e_{\mathcal{I}', A}$
  and $\iota_B \circ \partial^e_{\mathcal{I}, B} = \partial^e_{\mathcal{I}',
  B}$ since $\iota$ is a morphism of intervals, hence
  $\partial^e_{\mathcal{I}', A}$ and $\partial^e_{\mathcal{I}', B}$ are
  $\mathcal{I}$-trivial cofibrations. It follows that $t_1$ in \

%   \begin{center}
%     \epsfig{file=rev2-22.eps}
%   \end{center}
  
  \begin{center}
$
\xy
{\ar@{>->}^j (0,0)*+++{A} ; (35,0)*+{B}};
{\ar@{>->}_{\partial^e_{\mathcal{I}',A}} (0,0)*+++{A} ; (0,-30)*+{I(A)}};
{\ar@{>->}_{I'(j)} (0,-30)*+++{I(A)} ; (35,-30)*+{I(B)}};
{\ar@{>->}^{\partial^e_{\mathcal{I}',B}} (35,0)*+++{B} ; (35,-30)*+{I(B)}};
{\ar@{>.>}^t (22.5,-18)*+++{I(A) \vee B} ; (35,-30)*+{I(B)}};
{\ar@{>->}_{t_1} (35,0)*+++{B} ; (22.5,-18)*++{I(A) \vee B}};
{\ar@{>->}^{t_2} (0,-30)*+++{I(A)} ; (22.5,-18)*++{I(A) \vee B}};
{\ar@{-} (5,-5)*{} ; (5,-2.5)};
{\ar@{-} (5,-5)*{} ; (2.5,-5)};
\endxy
$
\end{center}
  
  is an $\mathcal{I}$-trivial cofibration and so is $t$ by remark
  $\ref{rem:unions}$. Hence $\Lambda^0_{\mathcal{I}'} \subseteq
  \mathcal{C}_{\mathcal{I}} \cap \mathcal{W}_{\mathcal{I}} $.
  
  Let $n > 0$ and suppose $\Lambda^{n - 1}_{\mathcal{I}'} \subseteq
  \mathcal{C}_{\mathcal{I}} \cap \mathcal{W}_{\mathcal{I}} $. Let $t \in
  \Lambda_{\mathcal{I}'}^{n - 1}$. Then $t + t$ is an $\mathcal{I}$-trivial
  cofibration and so is $k_2$ in

 %  \begin{center}
%     \epsfig{file=rev2-23.eps}
%   \end{center}
  
  \begin{center}
$
\xy
{\ar@{>->}^{t+t} (0,0)*+++{A+A} ; (35,0)*+{B+B}};
{\ar@{>->}_{\left [ \partial^0_{\mathcal{I}',A},\partial^1_{\mathcal{I}',A}\right ]} (0,0)*+++{A+A} ; (0,-30)*+{I'(A)}};
{\ar@{>->}_{I'(t)} (0,-30)*+++{I'(A)} ; (35,-30)*+{I'(B)}};
{\ar@{>->}^{\left [ \partial^0_{\mathcal{I}',B},\partial^1_{\mathcal{I}',B}\right ]} (35,0)*+++{B+B} ; (35,-30)*+{I'(B)}};
{\ar@{>.>}^{k} (22.5,-18)*+++{I'(A) \vee B+B} ; (35,-30)*+{I'(B)}};
{\ar@{>->}_{k_1} (35,0)*+++{B+B} ; (22.5,-18)*++{I'(A) \vee B+B}};
{\ar@{>->}^{k_2} (0,-30)*+++{I'(A)} ; (22.5,-18)*++{I'(A) \vee B+B}};
{\ar@{-} (6,-6)*{} ; (6,-3.5)};
{\ar@{-} (6,-6)*{} ; (3.5,-6)};
{\ar@{-} (17,-11.5)*{} ; (19.5,-11.5)};
{\ar@{-} (17,-11.5)*{} ; (17,-13.5)};
\endxy
$
\end{center}
  
  On the other hand, chasing around

%   \begin{center}
%     \epsfig{file=rev2-24.eps}
%   \end{center}
  
  \begin{center}
$
\xy
{\ar@{>->}^t (0,0)*+++{A} ; (25,0)*+{B}};
{\ar@{>->}_{\partial^e_{\mathcal{I}',A}} (0,0)*+++{A} ; (0,-20)*+{I'(A)}};
{\ar@{>->}_{I'(t)} (0,-20)*+++{I'(A)} ; (25,-20)*+{I'(B)}};
{\ar@{>->}^{\partial^e_{\mathcal{I}',B}} (25,0)*+++{B} ; (25,-20)*+{I'(B)}};
\endxy
$
\end{center}
  
  one finds that $I' (t)$ is an $\mathcal{I}$-weak equivalence so
  $k$ is an $\mathcal{I}$-trivial cofibration by remark \ref{rem:unions}.
  
  Hence
  \[ \Lambda^n_{\mathcal{I}'} \subseteq \mathcal{C}_{\mathcal{I}} \cap
     \mathcal{W}_{\mathcal{I}} \]
  for all $n \geqslant 0$. 
\end{proof}

\begin{proposition}
  \label{prop:weaki1}Suppose there is a morphism of intervals $\mathcal{I}
  \longrightarrow \mathcal{I}'$ which is a section-wise $\mathcal{I}$-weak
  equivalence. Then
  \[ \mathcal{W}_{\mathcal{I}'} \subseteq \mathcal{W}_{\mathcal{I}} \]
  and $\tmop{id}_{\mathcal{E}} : ( \mathcal{E}, \mathcal{I}) \longrightarrow (
  \mathcal{E}, \mathcal{I}')$ is a right Quillen functor.
\end{proposition}

\begin{proof}
  We have $\Lambda_{\mathcal{I}'} \subseteq \mathcal{C}_{\mathcal{I}} \cap
  \mathcal{W}_{\mathcal{I}}$ by lemma \ref{lem:weaki1}. Let $X \in
  \mathcal{E}$ be $\mathcal{I}$-naively fibrant. Then $!_X \in
  \mathcal{F}_{\mathcal{I}}$ by proposition \ref{prop:cisinski1}(i) so $!_X$
  has the right lifting property with respect to all $t \in
  \Lambda_{\mathcal{I}'}$. It follows that
  \[ \mathcal{E}^{\tmop{nf}}_{\mathcal{I}} \subseteq
     \mathcal{E}^{\tmop{nf}}_{\mathcal{I}'} \]
  and thus $\mathcal{W}_{\mathcal{I}'} \subseteq \mathcal{W}_{\mathcal{I}}$.
  Now both model structures have the same cofibrations, namely the monos.
  Hence $\tmop{id}_{\mathcal{E}} : ( \mathcal{E}, \mathcal{I}')
  \longrightarrow ( \mathcal{E}, \mathcal{I})$ preserves cofibrations and
  trivial cofibrations and is thus left Quillen.
\end{proof}

\begin{lemma}
  \label{lem:weaki2}Suppose there is a morphism of intervals $\mathcal{I}
  \longrightarrow \mathcal{I}'$ which is a section-wise $\mathcal{I}'$-weak
  equivalence. Then
  \[ \Lambda_{\mathcal{I}} \subseteq \mathcal{C}_{\mathcal{I}'} \cap
     \mathcal{W}_{\mathcal{I}'} \]
\end{lemma}

\begin{proof}
  Same argument as for lemma \ref{lem:weaki1}, save for a different case of
  the 2-of-3 property.
\end{proof}

\begin{proposition}
  \label{prop:weaki2}Suppose there is a morphism of intervals $\mathcal{I}
  \longrightarrow \mathcal{I}'$ which is a section-wise $\mathcal{I}'$-weak
  equivalence. Then
  \[ \mathcal{W}_{\mathcal{I}} \subseteq \mathcal{W}_{\mathcal{I}'} \]
  and $\tmop{id} : ( \mathcal{E}, \mathcal{I}) \longrightarrow ( \mathcal{E},
  \mathcal{I}')$ is a left Quillen functor.
\end{proposition}

\begin{proof}
  Same argument as for proposition \ref{prop:weaki1}.
\end{proof}

\section{Directed homotopy theories\label{sec:diho}}

\subsection{Dihomotopy}Let $\Delta_d$ be the unit interval
equipped with the discrete order, as in example \ref{exa:1}.

\begin{remark}
  \label{rem:interval-inclusions}For any $P \in \mathbb{P}$, let
 \[  \begin{array}{llll}
          k^e_{d, P} : & P & \longrightarrow & \Delta_d\\
          & x & \longmapsto & e
        \end{array}  \] 
  be the constant dimaps with values $e = 0$ and $e = 1$, respectively.
  Let
  \[ I_d  \overset{\tmop{def} .}{=} (-) \times h_{\mathbb{P} } (\Delta_d) :
     \tmop{Sh} ( \mathbb{P}, \tau) \longrightarrow \tmop{Sh} ( \mathbb{P}
     \tau) \]
  be the endofunctor acting by taking the product with $h_{\mathbb{P} }
  (\Delta_d)$. There are natural transformations
  \[  \begin{array}{l}
       \partial^e_d : \tmop{id}_{\tmop{Sh} ( \mathbb{P}, \tau)}
       \Longrightarrow I_d
     \end{array}  \]
  for $e \in \{0, 1\}$, given by
  \[ \begin{array}{llll}
       \partial^e_{d, F} : & F & \Longrightarrow & F \times h_{\mathbb{P} }
       (\Delta_d)\\
       &  &  & \\
       \partial^e_{d, F, P} : & F (P) & \longrightarrow & F (P) \times
       \mathbb{L} \left( P, \Delta_d \right)\\
       & x & \longmapsto & (x, k^e_{d, P})
     \end{array} \]
  There is furthermore the natural transformation $\sigma_d : I
  \Longrightarrow \tmop{id}_{\tmop{Sh} ( \mathbb{P}, \tau)}$ given
  components by the first projection $\sigma_{d, F}  \overset{\tmop{def}
  .}{=} \pi_{1, F} : F \times h_{\mathbb{P} } (\Delta_d) \Longrightarrow F$.
  Obviously, the quadruple $(I_d, \partial^0_d, \partial^1_d, \sigma_d)$ is a
  cylinder.
\end{remark}

\begin{lemma}
  \label{lem:cart-cylinder}The cylinder $\text{$(I_d, \partial^0_d,
  \partial^1_d, \sigma_d)$}$ is cartesian.
\end{lemma}

\begin{proof}
  Let $F, K, L \in \tmop{Sh} \left( \mathbb{P}, \tau \right)$, $\beta : F
  \Longrightarrow L$, $\gamma : F \Longrightarrow K \times h_{\mathbb{P} }
  (\Delta_d)$, $\alpha : K \Longrightarrow L$ and $e \in \{0, 1\}$. Suppose
  $\alpha$ is mono and the outer diagram of

%   \begin{center}
%     \epsfig{file=rev2-25.eps}
%   \end{center}
  
  \begin{center}
$
\xy
{\ar@{>->}^{\alpha_P} (0,0)*+++{K(P)} ; (45,0)*+{L(P)}};
{\ar@{>->}_{\partial_{d,K,P}^e} (0,0)*+++{K(P)} ; (0,-20)*+{K(P)\times \mathbb{L}(P,\Delta_d)}};
{\ar@{>->}_{\alpha_P\times\operatorname{id}} (0,-20)*+++{K(P)\times \mathbb{L}(P,\Delta_d)} ; (45,-20)*+{L(P)\times\mathbb{L}(P,\Delta_d)}};
{\ar@{->}^{\partial_{d,L,P}^e} (45,0)*+{L(P)} ; (45,-20)*+{L(P)\times\mathbb{L}(P,\Delta_d)}};
%{\ar@/^1.6pc/@{->}^{\sigma_{d,K,P}} (45,0)*+{L(P)} ; (0,0)*+{K(P)}};
{\ar@/_1.4pc/@{->}_{\sigma_{d,K,P}} (2,-17) ; (2,-3)};
{\ar@/_1.9pc/@{->}_{\gamma_{P}} (-15,15)*+{F(P)} ; (0,-20)*+{K(P)\times \mathbb{L}(P,\Delta_d)}};
{\ar@/^1.9pc/@{->}^{\beta_{P}} (-15,15)*+{F(P)} ; (45,0)*+{L(P)}};
{\ar@{.>}^{\sigma_{d,K,P}\;\circ\;\gamma_{P}} (-15,15)*+{F(P)} ; (0,0)*+{K(P)}};
\endxy
$
\end{center}
  
  commutes. Assume $x \in F (P)$, $(a, f) = \gamma_P (x)$ and $b =
  \beta_P (x)$. Then $f = k^e_{d,}$ and $b = \alpha_P (x)$ since the outer
  diagram commutes. Hence
  \[ \left( \partial^e_{d, K, P} \circ \sigma_{d, K, P} \circ \gamma_P \right)
     (x) = \left( \partial^e_{d, K, P} \circ \sigma_{d, K, P}  \right) \left(
     a, k^e_{d, P}  \right) = \partial^e_{d, K, P} (a) = \left( a, k^e_{d, P} 
     \right) = \gamma_P (x) \]
  and
  \[ \left( \alpha_P \circ \sigma_{d, K, P} \circ \gamma_P \right) (x) =
     \alpha_P (a) = b = \beta_P (x) \]
  Moreover, $ \sigma_{d, K, P} \circ \gamma_P$ is the unique morphism with
  this property since $\alpha_p$ is mono.
\end{proof}

\begin{lemma}
  \label{lem:canonical-mor}The canonical morphism $\left[ \partial^0_d,
  \partial^1_d \right] : \tmop{id}_{\tmop{Sh} ( \mathbb{P}, \tau)} +
  \tmop{id}_{\tmop{Sh} ( \mathbb{P}, \tau)} \Longrightarrow I_d$ is a mono.
\end{lemma}

\begin{proof}
  The components of this morphism at $F \in \tmop{Sh} ( \mathbb{P}, \tau)$ and
  $P \in \mathbb{P}$ are given by
  \[ \begin{array}{llll}
       \left[ \partial^0_{d, F, P}, \partial^1_{d, F, P} \right] : & F (P) + F
       (P) & \longrightarrow & F (P) \times \mathbb{P} \left( P, \Delta_d 
       \right)\\
       & m & \longmapsto & \left\{ \begin{array}{ll}
         (x, k^0_{d, P}) & m = \tmop{in}_1 (x)\\
         (x, k^1_{d, P}) & m = \tmop{in}_2 (y)
       \end{array} \right.\\
       &  &  & 
     \end{array} \]
  Suppose $\left[ \partial^0_{d, F, P}, \partial^1_{d, F, P} \right] (m) =
  \left[ \partial^0_{d, F, P}, \partial^1_{d, F, P} \right] (m')$. There are
  two possible cases:
  \begin{enumeratealpha}
    \item \label{case-a}$m = \tmop{in}_1 (x) \tmop{and} m' = \tmop{in}_1
    (x')$;
    
    \item \label{case-b}$m = \tmop{in}_2 (x) \tmop{and} m' = \tmop{in}_2 (x')$
  \end{enumeratealpha}
  for some $x, x' \in F (P)$. Hence
  \begin{enumeratealpha}
    \item $(x, k^0_{d, P}) = (x', k^0_{d, P}) \Longrightarrow x = x'$;
    
    \item $(x, k^1_{d, P}) = (x', k^1_{d, P}) \Longrightarrow x = x'$.
  \end{enumeratealpha}
\end{proof}

\begin{proposition}
  \label{prop:dihomo}The quadruple
  \[  \mathcal{I}_d  \overset{\tmop{def} .}{=} (I_d, \partial_d^0,
     \partial^1_d, \sigma_d) \]
  is an interval.
\end{proposition}

\begin{proof}
  The functor I preserves monos by construction. It preserves colimits by
  construction as well, since colimits are universal in a topos. Proposition
  \ref{prop:dihomo} follows thus by remark \ref{rem:interval-inclusions} and
  lemmata \ref{lem:cart-cylinder} and \ref{lem:canonical-mor}.
\end{proof}

\subsection{D-homotopy}Let $\Delta_D$ be be the unit interval equipped with
the natural total order, as in example \ref{exa:1}.

\begin{remark}
  \label{rem:natord}For any $P \in \mathbb{P}$, let
   \[  \begin{array}{llll}
          k^e_{D, P} : & P & \longrightarrow & \Delta_D\\
          & x & \longmapsto & e
        \end{array}  \] 
  be the the constant dimaps with values $e = 0$ and $e = 1$, respectively.
  Let
  \[ I_D  \overset{\tmop{def} .}{=} (-) \times h_{\mathbb{P} } (\Delta_D) :
     \tmop{Sh} ( \mathbb{P}, \tau) \longrightarrow \tmop{Sh} ( \mathbb{P}
     \tau) \]
  be the endofunctor acting by taking the product with $h_{\mathbb{P} }
  (\Delta_D)$. There are natural transformations
  \[  \begin{array}{l}
       \partial^e_D : \tmop{id}_{\tmop{Sh} ( \mathbb{P}, \tau)}
       \Longrightarrow I_D
     \end{array}  \]
  for $e \in \{0, 1\}$, given by
  \[ \begin{array}{llll}
       \partial^e_{D, F} : & F & \Longrightarrow & F \times h_{\mathbb{P} }
       (\Delta_D)\\
       &  &  & \\
       \partial^e_{D, F, P} : & F (P) & \longrightarrow & F (P) \times
       \mathbb{L} \left( P, \Delta_D \right)\\
       & x & \longmapsto & (x, k^e_{D, P})
     \end{array} \]
  There is furthermore the natural transformation $\sigma_D : I
  \Longrightarrow \tmop{id}_{\tmop{Sh} ( \mathbb{P}, \tau)}$ given
  component wise by the first projection $\sigma_{D, F}  \overset{\tmop{def}
  .}{=} \pi_{1, F} : F \times h_{\mathbb{P} } (\Delta_d) \Longrightarrow F$. 
\end{remark}

\begin{proposition}
  \label{prop:d-homo}The quadruple
  \[  \mathcal{I}_D  \overset{\tmop{def} .}{=} (I_D, \partial_D^0,
     \partial^1_D, \sigma_D) \]
  is an interval.
\end{proposition}

\subsection{Dihomotopy vs. D-homotopy}

\begin{proposition}
  \label{prop:obang}Let $0^D : \tmmathbf{1}_{\mathbb{L}} \longrightarrow
  \Delta_D$ the dimap choosing $0$. The induced morphism of sheaves
  \[ 0^D_{\ast} : \tmmathbf{1} \longrightarrow h_{\mathbb{P} } \left( \Delta_D
     \right) \]
  is an $\mathcal{I}_d$-homotopy inverse of the canonical morphism $!_D :
  h_{\mathbb{P} } \left( \Delta_D \right) \longrightarrow \tmmathbf{1}$. In
  particular, $h_{\mathbb{P} } \left( \Delta_D \right)$ is (strongly)
  $\mathcal{I}_d$-contractible.
\end{proposition}

\begin{proof}
  Obviously $!_D \circ 0^D_{\ast} = \tmop{id}$. Let $f \in h_{\mathbb{P} }
  \left( \Delta_D \right) (P) = \mathbb{L} (P, \Delta_D)$ and $k \in
  h_{\mathbb{P} } \left( \Delta_d \right) (P) = \mathbb{L} (P, \Delta_d)$. The
  assignment
  \[ h_P (f, g) \overset{\tmop{def} .}{=} f \cdot g \]
  (with $f \cdot g$ the point-wise multiplication) determines the morphism of
  sheaves
  \[ h : h_{\mathbb{P} } \left( \Delta_D \right) \times h_{\mathbb{P} } \left(
     \Delta_d \right) \longrightarrow h_{\mathbb{P} } \left( \Delta_D \right)
  \]
  This morphism makes

%   \begin{center}
%     \epsfig{file=rev2-26.eps}
%   \end{center}
  
\begin{center}
$
\xymatrix{
\mathbb{L}(P,\Delta_D) \ar[d]_{\partial_{d,h_\mathbb{P}(\Delta_D),P}^{0}} 
\ar[dr]^{0_\ast^D \circ !_D} &
\\
\mathbb{L}(P,\Delta_D) \times \mathbb{L}(P,\Delta_d) 
\ar[r]^>>>>>>{h_P} &
\mathbb{L}(P,\Delta_D)
\\
\mathbb{L}(P,\Delta_D) \ar[u]^{\partial_{d,h_\mathbb{P}(\Delta_D),P}^{1}} 
\ar@{=}[ur] &
}
$
\end{center}
  
  commute for all $P \in \mathbb{P}$, so $h$ is an \
  $\mathcal{I}_d$-homotopy witnessing $0^D_{\ast} \circ !_D \sim \tmop{id}$.
\end{proof}

\begin{corollary}
  \label{cor:obang} The morphism of sheaves
  \[ \tmop{id}_F \times 0^D_{\ast} : F \times \tmmathbf{1} \longrightarrow F
     \times h_{\mathbb{P} } \left( \Delta_D \right)  \]
  is an $\mathcal{I}_d$-weak equivalence for all $F \in \tmop{Sh} \left(
  \mathbbm{P}, \tau \right)$.
\end{corollary}

\begin{proof}
  We have
  \[ \text{$\left( \tmop{id}_F \times !_D \right) \circ \left( \tmop{id}_F
     \times 0^D_{\ast} \right) = \tmop{id}_{F \times \tmmathbf{1}}$} \]
  and
  \[ \left( \tmop{id}_F \times 0^D_{\ast} \right) \circ \left( \tmop{id}_F
     \times !_D \right) \sim \tmop{id}_{F \times h_{\mathbb{P} } \left(
     \Delta_D \right)} \]
  by functoriality of $F \times (-)$.
\end{proof}

\begin{lemma}
  \label{lem:id-bang}Let $!_d : h_{\mathbb{P} } \left( \Delta_d \right)
  \longrightarrow \tmmathbf{1}$ be the canonical morphism of sheaves. The
  morphism of sheaves $(\tmop{id}_F \times !_d) : F \times h_{\mathbb{P} }
  \left( \Delta_d \right) \longrightarrow F \times \tmmathbf{1}$ is an
  $\mathcal{I}_d$-weak equivalence for all $F \in \tmop{Sh} \left(
  \mathbbm{P}, \tau \right)$. 
\end{lemma}

\begin{proof}
  The square

%   \begin{center}
%     \epsfig{file=rev2-27.eps}
%   \end{center}
  
  \begin{center}
$
\xy
{\ar@{->}^>>>>>>>{\sigma \: = \: \pi_1} (0,0)*+{F \times h_\mathbb{P}\left ( \Delta_d \right )} ; (25,0)*+{F}};
{\ar@{->}_{id \times !} (0,0)*+{F \times h_\mathbb{P}\left ( \Delta_d \right )} ; (0,-20)*+{F \times \mathbf{1}}};
{\ar@{->}^\cong (0,-20)*+{F \times \mathbf{1}} ; (25,-20)*+{F}};
{\ar@{=} (25,0)*+{F} ; (25,-20)*+{F}};
\endxy
$
\end{center}
  
  commutes and $\sigma$ is an $\mathcal{I}_d$-weak equivalence by
  remark \ref{rem:homoequs2}.
\end{proof}

\begin{remark}
  \label{rem:iota}Let $i : \Delta_d \longrightarrow \Delta_D$ be the morphism
  in $\mathbb{L}$ with the identity as its underlying map and let
  \[ \text{$i_{\ast} : h_{\mathbb{P} } \left( \Delta_d \right)
     \longrightarrow h_{\mathbb{P} } \left( \Delta_D \right)$} \]
  be the induced morphism of sheaves. The morphism of sheaves $\iota : I_d
  \longrightarrow I_D$ given by
  \[ \begin{array}{l}
       
     \end{array} \iota_F  \overset{\tmop{def} .}{=} \tmop{id}_F \times
     i_{\ast}  \]
  at $F \in \tmop{Sh} ( \mathbb{P}, \tau)$ is a morphism of intervals $\iota :
  \mathcal{I}_d \longrightarrow \mathcal{I}_D$.
\end{remark}

\begin{lemma}
  \label{lem:itoiprime}The morphism of intervals $\iota : \mathcal{I}_d
  \longrightarrow \mathcal{I}_D$ is a component-wise $\mathcal{I}_d$-weak
  equivalence.
\end{lemma}

\begin{proof}
  The assignment
  \[ h'_P (f, g) \overset{\tmop{def} .}{=} f \cdot g \]
  determines a morphism of sheaves
  \[ h' : h_{\mathbb{P} } \left( \Delta_d \right) \times h_{\mathbb{P} }
     \left( \Delta_d \right) \longrightarrow h_{\mathbb{P} } \left( \Delta_D
     \right) \]
  such that

%   \begin{center}
%     \epsfig{file=rev2-28.eps}
%   \end{center}
  
  \begin{center}
$
\xymatrix{
F \times h_\mathbb{P}(\Delta_d) 
\ar[d]_{\partial_{d,F \times h_\mathbb{P}(\Delta_D)}^{0}} 
\ar[r]^{id_F \times !_d} & F \times \mathbf{1}
\ar[d]^{id_F \times 0_\ast^D}
\\
 F \times h_\mathbb{P}(\Delta_d) \times h_\mathbb{P}(\Delta_d)
\ar[r]^>>>>>{id_F \times h'} &
F \times h_\mathbb{P}(\Delta_D)
\\
F \times h_\mathbb{P}(\Delta_d) 
\ar[u]^{\partial_{d,F \times h_\mathbb{P}(\Delta_D)}^{1}} 
\ar[ur]_{\iota_F} &
}
$
\end{center}
  
  commutes for all $F \in \tmop{Sh} ( \mathbb{P}, \tau)$. Now we
  have that $\tmop{id}_F \times !_d$ is an $\mathcal{I}_d$-weak equivalence by
  lemma \ref{lem:id-bang} while $\tmop{id}_F \times 0^D_{\ast}$ is an
  $\mathcal{I}_d$ - weak equivalence by corollary \ref{cor:obang}, so
  $\iota_F$ is an $\mathcal{I}_d$-weak equivalence by remark
  \ref{rem:homoequs2}.
\end{proof}

\begin{theorem}
  \label{theo:dih-model-structure} $\mathcal{W}_{\mathcal{I}_D} \subseteq
  \mathcal{W}_{\mathcal{I}_d}$. In particular, \ $\tmop{id} : \left( \tmop{Sh}
  ( \mathbb{P}, \tau), \mathcal{I}_D \right) \longrightarrow \left( \tmop{Sh}
  ( \mathbb{P}, \tau), \mathcal{I}_d \right) $ is a left Quillen functor.
\end{theorem}

\begin{proof}
  By propositions \ref{prop:weaki1} and lemma \ref{lem:itoiprime}. 
\end{proof}

\begin{remark}
  The ``contracting homotopy''
  \[ h : h_{\mathbb{P} } \left( \Delta_d \right) \times h_{\mathbb{P} }
     \left( \Delta_d \right) \longrightarrow h_{\mathbb{P} } \left( \Delta_D
     \right) \]
  of proposition \ref{prop:obang}, given by the assignment
  \[ h_P (f, g) \overset{\tmop{def} .}{=} f \cdot g \]
  is crucial to establish theorem \ref{theo:dih-model-structure}. This
  homotopy exhibits the canonical morphism
  \[ !_D : h_{\mathbb{P} } \left( \Delta_D \right) \longrightarrow
     \tmmathbf{1} \]
  as an $\mathcal{I}_d$-homotopy equivalence (hence as an $\mathcal{I}_d$-weak
  equivalence). The argument fails in the other direction since there is no
  contracting homotopy
  \[ h_{\mathbb{P} } \left( \Delta_D \right) \times h_{\mathbb{P} } \left(
     \Delta_D \right) \longrightarrow h_{\mathbb{P} } \left( \Delta_d \right)
  \]
  as dimaps to $\Delta_d$ have to be constant on order-connected components. We
  nonetheless conjecture that
  \[  \mathcal{W}_{\mathcal{I}_d}  \not\subseteq
     \mathcal{W}_{\mathcal{I}_D} \]
  and expect to prove the assertion by cohomological means. The latter will be
  described in a subsequent paper.
\end{remark}


\begin{thebibliography}{Bek01}
  \bibitem[Bek01]{beke:2001}Tibor Beke. {\newblock}Sheafifiable homotopy model
  categories. {\newblock}\textit{Mathematical Proceedings of the Cambridge
  Philosophical Society}, 129(03):447--475, 2001.
  
  \bibitem[BW06]{buwo}Peter Bubenik and Krzysztof Worytkiewicz. {\newblock}A
  model category for local po-spaces. {\newblock}\textit{Homology, Homotopy
  and Applications}, 8(1), 2006.
  
  \bibitem[Cis02]{cisinski-topos}Denis-Charles Cisinski. {\newblock}Th\'eories
  homotopiques dans les topos. {\newblock}\textit{Journal of Pure and
  Applied Algebra}, 174:43--82, 2002.
  
  \bibitem[FRG06]{FRG}Lisbeth Fajstrup, Martin Raussen and Eric Goubault.
  {\newblock}Algebraic topology and concurrency.
  {\newblock}\textit{Theoretical Computer Science}, 357:241--278, 2006.
  
  \bibitem[Gra03]{grandis2}Marco Grandis. {\newblock}Directed homotopy theory.
  {\newblock}\textit{Cahiers de Topologie et G\'eom\'etrie Diff\'erentielle
  Cat\'egorique}, 44:281--316, 2003.
  
  \bibitem[GZ67]{gz}Peter Gabriel and Michel Zisman.
  {\newblock}\textit{Calculus of Fractions and Homotopy Theory}.
  {\newblock}Springer, 1967.
  
  \bibitem[Jar06]{cahoth}John~F. Jardine. {\newblock}Categorical homotopy
  theory. {\newblock}\textit{Homology, Homotopy and Applications}, 8, 2006.
  
  \bibitem[To\"e]{toen-dea}Bertrand To\"en. {\newblock}Champs alg\'ebriques.
  {\newblock}Lecture notes.

\end{thebibliography}
\end{document}